%% file: BBGOV-gometric_conditions.tex
\documentclass[11pt,reqno]{amsart}

\usepackage{amsfonts,amsthm,latexsym,amsmath,amssymb,amscd,amsmath, epsf}
\usepackage[letterpaper,margin=1in]{geometry}
\usepackage{color}
\usepackage[utf8]{inputenc}
\usepackage{picture}
\usepackage{tikz}
\usepackage{graphicx}
\usepackage{verbatim}
\usepackage{tcolorbox}
\tcbuselibrary{skins}
\usepackage{ulem}
\input{mypreamble.tex}

 \newtheorem{theorem}{Theorem}[section]
 \newtheorem{corollary}[theorem]{Corollary}
 \newtheorem{lemma}[theorem]{Lemma}
 \newtheorem{proposition}[theorem]{Proposition}
 \newtheorem{conjecture}[theorem]{Conjecture}
  \newtheorem*{question*}{\sc Question}
  \newtheorem*{problem}{\sc Problem} 
 \theoremstyle{definition}
 \newtheorem{definition}[theorem]{Definition}
 \theoremstyle{remark}
 \newtheorem{remark}[theorem]{Remark}
 \theoremstyle{definition}
 \newtheorem{example}[theorem]{Example}
 \newtheorem{notation}[theorem]{\it Notation}

\theoremstyle{plain}
\newtheorem*{claim*}{Claim}

 \usepackage{listings}

\newcommand{\Flat}{\mathrm{Flat}}
\newcommand{\rk}{\mathrm{rank}}
\newcommand{\cat}{\mathrm{cat}}
\newcommand{\rank}{\mathrm{R}}
\newcommand{\brank}{\mathrm{\underline{R}}}
\renewcommand{\HF}{\mathrm{HF}}

 \title[Strict submultiplicativity of rank and border rank]{Geometric conditions for strict submultiplicativity of rank and border rank}
 \dedicatory{To Giorgio Ottaviani, on the occasion of his 60th birthday}

\subjclass[2010]{15A69; 14N05, 14H99}
\keywords{rank, border rank, tensor product, Segre product, secant variety}
 
 \author[E. Ballico]{Edoardo Ballico}
 \author[A. Bernardi]{Alessandra Bernardi}

 \address[E. Ballico, A. Bernardi, A. Oneto]{Dipartimento di Matematica, Universit\`a di Trento, 38123 Povo (TN), Italy}

\author[F. Gesmundo]{Fulvio Gesmundo}
\address[F. Gesmundo]{QMATH, Dept. Math. Sciences, U. Copenhagen, Universitetsparken 5, 2100 Copenhagen O., Denmark}

\author[E. Ventura]{Emanuele Ventura}
\address[E. Ventura]{Universit\"{a}t Bern, Mathematisches Institut, Sidlerstrasse 5, 3012 Bern, Switzerland}

\email[Ballico]{edoardo.ballico@unitn.it }
\email[Bernardi]{alessandra.bernardi@unitn.it}
\email[Gesmundo]{fulges@math.ku.dk}
\email[Oneto]{alessandro.oneto@unitn.it}
\email[Ventura]{emanueleventura.sw@gmail.com, emanuele.ventura@math.unibe.ch}

\newcommand{\uR}{\underline{\mathrm{R}}}

\newcommand{\Alb}{\mathrm{Alb}}
 
\begin{document}
  
\begin{abstract}
The $X$-rank of a point $p$ in projective space is the minimal number of points of an algebraic variety $X$ whose linear span contains $p$. This notion is naturally submultiplicative under tensor product. We study geometric conditions that guarantee strict submultiplicativity. We prove that in the case of points of rank two, strict submultiplicativity is entirely characterized in terms of the trisecant lines to the variety. Moreover, we focus on the case of curves: we prove that for curves embedded in an even-dimensional projective space, there are always points for which strict submultiplicativity occurs, with the only exception of rational normal curves.
\end{abstract}

\maketitle

 \section{Introduction}\label{sec:intro}
 
Let $V$ be a complex finite dimensional vector space and let $\bbP V$ be the corresponding projective space. When $V = \bbC^{n+1}$, denote $\bbP \bbC^{n+1}$ simply by $\bbP^n$. Throughout the paper, a projective, irreducible, reduced and linearly nondegenerate (i.e., not contained in a hyperplane) algebraic variety is called \emph{variety}.

Ever since the nineteenth century, a line of research dealt with determining normal forms of algebraic objects in terms of basic building blocks: a classical example concerns expressions of homogeneous polynomials as sum of powers of linear forms. These additive problems can be rephrased equivalently in geometric terms as follows: given a variety and a point in its ambient space, determine sets of points of the variety whose linear span contains the given point. This approach motivated the study of secant varieties during the twentieth century. In the last decades the connections with applications involving additive tensor decomposition attracted the interest of a broad community, both in pure and applied mathematics and in other fields. In the rich literature, we briefly mention the classical \cite{Clebsch,Sylv:PrinciplesCalculusForms,Palatini1}, concerning the study of homogeneous polynomials, \cite{Palatini:SuperficieAlg,Adl:JoinsHigherSecantVarieties}, studying secant varieties of curves, \cite{AllRho:PhylogeneticIdealsVarsGeneralMarkovModel,Strassen:Gauss_elimination_not_optimal,DurVidCir:TrheeCubitsEntTwoIneqWays} drawing connection with phylogenetics, theoretical computer science and quantum information theory. We refer to \cite{Lan12:Book,BCCGO:Hitchhiker} and the references therein for a more extensive presentation.

We formally introduce the notion of rank with respect to an arbitrary variety $X\subseteq \bbP V$. Given a point $p \in \bbP V$, the \textbf{$X$-rank} (or simply the \emph{rank}) of $p$ is the minimal number of points of $X$ whose linear span contains $p$:
\[
\rank_X(p) := \min\left\{r ~|~ \exists~q_1,\ldots,q_s \in X,~ p \in \langle q_1,\ldots,q_s\rangle\right\}.
\]
 
Let $\sigma_r(X)$ be the \textbf{$r$-th secant variety} of $X$, i.e., the Zariski-closure of the set of points whose $X$-rank is at most $r$. The \textbf{border $X$-rank} (or simply the \emph{border rank}) of $p$ is 
\[
\brank_X(p) := \min\left\{r ~|~ p \in \sigma_r(X)\right\}. 
\]

It is natural to study properties of rank and border rank with respect to basic operations among varieties. Due to the additive nature of the problem, in the most general framework one is interested in relations between the rank of a linear combination of two points and their ranks. For example, in the tensor setting, Strassen \cite{Str:VermeidungDiv} conjectured that the rank of a direct sum of two tensors always coincides with the sum of their ranks: this conjecture has been answered affirmatively in several cases \cite{CCG,BGL13:DeterminantalEquations,CCC15,Teit:SuffCondStrassen,CCCGW15,LM17:AbelianTensors,CarCatOne:WaringLoci,BucPosRup:StrassenAdditivity}. In the case of tensors with three factors, a counterexample was given by Shitov \cite{Shitov:StrassenCounterexample}, while the analogous equality for border ranks was shown to be false already by Sch\"onhage \cite{Sch81:Partial}. In theoretical computer science, one is interested in multiplicativity properties of tensor rank and border rank under Kronecker powers, which capture the asymptotic complexity of the bilinear map defined by the tensor \cite{Str:RelativeBilComplMatMult}. 

In this work, we are interested in multiplicativity properties of rank and border rank under tensor product. Given two varieties $X_1 \subseteq \bbP V_1$ and $X_2 \subseteq \bbP V_2$, their \textit{Segre product} is the image of $X_1 \times X_2 \subseteq \bbP V_1 \times \bbP V_2$ under the \emph{Segre embedding} 
\[
\begin{array}{c c c}
 \bbP V_1 \times \bbP V_2 & \longrightarrow & \bbP (V_1\otimes V_2), \\
([v_1] , [v_2]) & \mapsto & [v_1 \otimes v_2].
 \end{array}
\]
Denote the Segre product of $X_1$ and $X_2$ by $X_1 \times X_2$. We often identify points in projective space and vectors of the line they represent; in particular, we will drop the bracket $[-]$ from the notation.

For any $p_1 \in \bbP V_1$ and $p_2 \in \bbP V_2$, one has
 \begin{equation}\label{trivial upper bounds}
 \rank_{X_1 \times X_2}(p_1\otimes p_2) \leq \rank_{X_1}(p_1)\rank_{X_2}(p_2) \text{ \quad and \quad} \brank_{X_1 \times X_2}(p_1\otimes p_2) \leq \brank_{X_1}(p_1)\brank_{X_2}(p_2).
 \end{equation}
Certain techniques to determine lower bounds on rank and border rank guarantee that the lower bound propagates to the tensor product and can be used to prove multiplicativity: this is the case of flattening lower bounds, see Section \ref{section: homogeneous poly}. However, both inequalities in \eqref{trivial upper bounds} can be strict in general, as shown in \cite{ChrJenZui:NonMultTensorRank} for rank and \cite{ChrGesJen:BorderRankNonMult} for border rank.

Despite the achievements mentioned above originated from tensor problems, we investigate the multiplicativity problem in general. 

\begin{problem}\label{problem}
Determine geometric conditions which guarantee either multiplicativity or strict submultiplicativity in the inequalities in \eqref{trivial upper bounds}.
\end{problem}

{\bf Contributions and structure of the paper.}

\begin{itemize}
\item In Section \ref{sec:multisecant_spaces}, we classify the ranks of tensor products of points of rank $2$: if $p_1,p_2$ are points of rank $2$ with respect to varieties $X_1, X_2$ respectively, we give sufficient and necessary conditions so that the $(X_1\times X_2)$-rank of $p_1 \otimes p_2$ is equal to $3$, instead of $4$; see Theorem~\ref{thm: rank 2 x rank 2 classification}. As a corollary, if $p$ has $X$-rank $2$, then $p^{\otimes 2}$ has $X^{\times 2}$-rank equal to $3$ if and only if $p$ lies on a \emph{multisecant line}. 

\item We show that if a variety $X$ admits a secant $r$-dimensional plane $\bbP W$ intersecting $X$ in more than $r+1$ points then, for any $p \in \bbP W$, the $X^{\times (r+1)}$-rank of $p^{\otimes (r+1)}$ is strictly less than $\rank_{X^{\times (r+1)}}(p)^{r+1}$; see Proposition~\ref{prop: Pr r+2 secant}.

\item In Section \ref{sec:geometric_cond}, we investigate conditions which guarantee that the geometric construction in \cite{ChrGesJen:BorderRankNonMult} may be applied; see Proposition \ref{prop: map to abelian implies not rational} and Proposition \ref{prop: projection generates multidrop}. We show that if $X \subseteq \bbP^{2k}$ is a curve which is not the rational normal curve of degree $2k$, then there are always examples of strict submultiplicativity; see Theorem \ref{thm: curves}.

\item In Section \ref{section: homogeneous poly}, we turn our attention to homogeneous polynomials. In more geometric terms, we characterize rank submultiplicativity when either $X$ is a rational normal curve or $X$ is the third Veronese embedding of $\bbP^2$; see Proposition \ref{thm:submult_binary} and Proposition \ref{prop:submult_planecubics}, respectively. 

\item In Section \ref{sec:product_decomp}, we consider cases for which multiplicativity of rank holds. We ask whether all minimal decompositions of products are products of minimal decompositions of the factors. In Theorem~\ref{thm: identifiability product via embeddings}, we give conditions which guarantee a positive answer while in Example \ref{ex:NonProdMinimalDecomp} we provide an example having a negative answer.
\end{itemize}

In view of these results, we propose the following.

\begin{conjecture}\label{conj}
Let $X \subseteq \bbP^n$ be a variety and let $p \in \bbP^n$. 
\begin{center}
If $\brank_{X}(p) < \rank_{X}(p)$, then $\rank_{X^{\times 2}}(p \otimes p) < \rank_{X}(p)^2$.
\end{center}
\end{conjecture}

Observe that Conjecture \ref{conj} holds \textit{asymptotically}, in the following sense. If $p \in \bbP^n$ is a point such that $\uR_{X}(p) < \rmR_{X}(p)$, then there is a value $k$ such that $\rmR_{X^{\times k}} \left( p^{\otimes k}\right) < \rmR_{X} ( p)^k$: this is a consequence of the fact that the two limits $\lim_{k \to \infty} \uR_{X^{\times k}} \left( p^{\otimes k}\right)^{1/k}$ and $\lim_{k \to \infty}\rmR_{X^{\times k}} \left( p^{\otimes k}\right)^{1/k}$ coincide and they are bounded from above by $\uR_X(p)$. We refer to \cite[Section 6]{ChrGesJen:BorderRankNonMult} for details on this asymptotic behaviour.

\subsection*{Acknowledgements} E.B. and A.B. acknowledge financial support from GNSAGA of INDAM (Italy). F.G. acknowledges financial support from the VILLUM FONDEN via the QMATH Centre of Excellence (Grant no. 10059). A.O. acknowledges financial support from the Alexander von Humboldt-Stiftung via a Humboldt Research Fellowship for Postdoctoral Researchers (April 2019 - March 2021) at OVGU Magdeburg (Germany).

This collaboration started while F.G., A.O. and E.V. were visiting University of Trento for a Research in Pairs program at CIRM Trento in July 2018, continued while A.B. and F.G. were visiting the Institute for Computational and Experimental Research in Mathematics in Providence, RI, in Fall 2018, and was completed while A.B., F.G. and A.O. were visiting University of Pavia for the XXI Congresso dell'Unione Matematica Italiana, in September 2019. We thank CIRM, ICERM and UMI for providing good research environments to work on this project.

\section{Multisecant spaces and strict submultiplicativity}\label{sec:multisecant_spaces}
In this section, we look for geometric conditions on a variety $X$ which guarantee strict submultiplicaitivity for points of $X$-rank $2$. The first idea comes from \cite[Lemma 4.1]{ChrGesJen:BorderRankNonMult} which shows that if the secant variety $\sigma_r(X)$ has a trisecant line $\bbP L$ such that one of the points of $\bbP L \cap \sigma_r(X)$ lies on $X$ itself, then there exists at least one point on $\bbP L$ for which the rank multiplicativity does not hold. We prove that strict rank submultiplicativity for points having $X$-rank $2$ depends on the existence of a trisecant line to $X$ itself.

\begin{definition}\label{def: multisecant line}
A line $\bbP L$ is \textbf{multisecant} to a variety $X \subset \bbP V$ if the intersection $X \cap \bbP L$ contains a set of at least $3$ distinct points.
\end{definition}

\begin{notation}
	Given a vector space $V$, denote by ${\rm Sym}^rV$ the space of symmetric elements of $V^{\otimes r}$. Let $\nu_r$ be the
	map from $V$ to ${\rm Sym}^rV$ which sends $v \in V$ to $v^{\otimes r} \in {\rm Sym}^rV$. Denote by $\nu_r$ also the corresponding map between the projective spaces known as Veronese embedding. Given a subset $S \subseteq V$, denote by $\langle S \rangle$ the linear span of $S$ in $V$. Similarly, given a subset $S \subseteq \bbP V$, denote by $\langle S \rangle$ the projective linear span of $S$ in $\bbP V$.
\end{notation}

\begin{proposition}\label{prop: every point on multisecant}
Let $X \subseteq \bbP V$ be a variety and let $\bbP L$ be a line such that $\bbP L \cap X$ contains a set of at least $k+1$ points. Let $p\in \bbP L\setminus X$. Then, for every $r \leq k$ 
 \[
\rmR_{X^{\times r}} \left(p^{\otimes r}\right) \leq r+1. 	
 \]
 In particular, for every $r \geq 2$, $\rmR_{X^{\times r}}(p^{\otimes r}) < \rmR_X(p)^r$. 
\end{proposition}
\begin{proof}
Notice that $\rmR_X(p) = 2$, because $p$ lies on a secant line to $X$ and $p \notin X$. Let $L \subset V$ be the vector space of dimension $2$ defining the projective line $\bbP L$: $p \in \bbP L$ and $p^{\otimes r} = \nu_r(p) \in \bbP ({\rm Sym}^r L)$. Fix $r \leq k$ and let $S \subseteq \bbP L \cap X$ be a set of $r+1$ points. Then, $\nu_r(S) \subseteq \nu_r(X) \subseteq X^{\times r}$ is a set of $r+1$ points lying on the rational normal curve $\nu_r(\bbP L) \subseteq \bbP({\rm Sym}^r L) \subseteq \bbP V^{\otimes r}$. In particular, $\nu_r(S)$ is a set of $r+1$ linearly independent points in $\bbP({\rm Sym}^r L) \simeq \bbP^r$; hence, $\langle \nu_r(S) \rangle = \bbP({\rm Sym}^r L)$. Since $p^{\otimes r} \in \langle \nu_r(S) \rangle$, we conclude that $\rank_{X^{\times r}} \left(p^{\otimes r}\right) \leq r+1$. 

The inequality $\rmR_{X^{\times r}}(p^{\otimes r}) < \rmR_X(p)^r$ for every $r \geq 2$ follows because if strict submultiplicativity holds for $r = 2$, as shown above, then it holds for any $r \geq 2$.
\end{proof}

\begin{remark}
More generally, the argument of Proposition \ref{prop: every point on multisecant} applies if $\bbP L \cap X$ is any \textit{$0$-dimensional scheme} rather than a set of points. The \textit{cactus $X$-rank} of a point $p$ is the minimal degree of a $0$-dimensional scheme on $X$ whose linear span contains $p$ (see \cite{RS11:RankSymmetricForm, BR}). Then, similarly as for the $X$-rank, there is a notion of \textit{cactus varieties} and \textit{border cactus $X$-rank} (see \cite{BR,BuczBucz:SecantVarsHighDegVeroneseReembeddingsCataMatAndGorSchemes}). Proposition \ref{prop: every point on multisecant} applies to cactus rank: if $\bbP L$ is a line whose intersection with $X$ is a $0$-dimensional scheme of degree at least $k+1$ then for every $r \leq k$, the cactus $X^{\times r}$-rank of $p^{\otimes r}$ is at most $r+1$.
\end{remark}

Proposition \ref{prop: every point on multisecant} guarantees that for every point $p$ lying on a multisecant line of $X \subseteq \bbP V$, but not on $X$, multiplicativity of rank does not hold; in particular, the $X^{\times 2}$-rank of $p^{\otimes 2}$ is at most~$3$. Theorem~\ref{thm: rank 2 x rank 2 classification} below shows that this is essentially the only way that the tensor product of two elements of rank $2$ has rank $3$ instead of $4$. 

First, we record an easy observation which will be used in the proof of Theorem \ref{thm: rank 2 x rank 2 classification}.

\begin{lemma}\label{lemma: spanning p1p2 implies factors span p1 and p2}
 Let $p_1 \otimes p_2 \in \bbP (V_1 \otimes V_2)$. Suppose $p_1 \otimes p_2 \in \langle a_1 \otimes b_1 \vvirg a_r \otimes b_r\rangle$. Then, $p_1 \in \langle a_1 \vvirg a_r\rangle$ and $p_2 \in \langle b_1 \vvirg b_r\rangle$.
\end{lemma}
\begin{proof}
 After passing to the underlying linear spaces, with a slight abuse of notation, we write $p_1 \otimes p_2 = \textsum_i \lambda_i a_i \otimes b_i \in V_1 \otimes V_2$. If $\beta \in V_2^*$ is a linear form such that $\beta(p_2) \neq 0$, then $\beta(p_2) p_1 = \textsum_i \lambda_i \beta(b_i) a_i$ is a linear combination of $a_1 \vvirg a_r$ which gives $p_1$. Analogously, one can prove that $p_2 \in \left\langle b_1,\ldots,b_r \right\rangle$.
\end{proof}

\begin{theorem}\label{thm: rank 2 x rank 2 classification}
 For $i = 1,2$, let $X_i \subseteq \bbP V_i$ be varieties and let $p_i \in \bbP V_i$ such that $\rmR_{X_i}(p_i) = 2$. Then, $3 \leq \rmR_{X_1 \times X_2} (p_1 \otimes p_2) \leq 4$. Moreover, for $a_1,a_2,a_3 \in X_1$ and $b_1,b_2,b_3 \in X_2$, the following are equivalent:
\begin{enumerate}[(i)]
 \item\label{(i)} $\rmR_{X_1 \times X_2}(p_1 \otimes p_2) = 3$ with $p_1 \otimes p_2 \in \langle a_1 \otimes b_1, a_2 \otimes b_2, a_3 \otimes b_3\rangle$;
 \item\label{(ii)} the linear spaces $\bbP L_1 =  \langle a_1,a_2,a_3 \rangle$ and $\bbP L_2 =  \langle b_1,b_2,b_3 \rangle$ are multisecant lines to $X_1$ and $X_2$, respectively, where the $a_i$'s and the $b_i$'s are all distinct; moreover, if $\phi: \bbP L_1 \to \bbP L_2$ is the unique linear map such that $\phi(a_j) = b_j$, then $\phi(p_1) = p_2$.
\end{enumerate}
\end{theorem}
\begin{proof}
 The upper bound $\rmR_{X_1 \times X_2} (p_1 \otimes p_2) \leq 4$ is immediate from submultiplicativity.
 
First, we show the lower bound $3 \leq \rmR_{X_1 \times X_2} (p_1 \otimes p_2)$. Since $p_1 \otimes p_2 \notin X_1 \times X_2$, we have $\rmR_{X_1 \times X_2}(p_1\otimes~p_2)\geq~2$. Suppose equality holds and $p_1 \otimes p_2 \in \langle a_1 \otimes b_1, a_2 \otimes b_2\rangle$, for $a_1,a_2 \in X_1$ and $b_1,b_2 \in X_2$. Let $\bbP L_1 =  \langle a_1 , a_2 \rangle$ and $\bbP L_2 = \langle b_1 ,b_2\rangle$: then $\dim \bbP L_1 = \dim \bbP L_2 = 1$ otherwise $p_1 = a_1 \in X_1$ or $p_2 = b_1 \in X_2$. Now, regard $p_1 \otimes p_2$, $a_1 \otimes b_1$ and $a_2 \otimes b_2$ as rank one matrices,
after a suitable choice of coordinates, they can be identified with 
 \[
   a_1 \otimes b_1 = \left( \begin{array}{cc}
                          1 & 0 \\ 0 & 0
                         \end{array} \right) \quad \text{ and } \quad
a_2 \otimes b_2 = \left( \begin{array}{cc}
                          0 & 0 \\ 0 & 1
                         \end{array} \right) ;
 \]
then, $\langle a_1 \otimes b_1 , a_2 \otimes b_2 \rangle$ does not contain any other rank one matrix, providing a contradiction. This shows $\rmR_{X_1 \times X_2}(p_1 \otimes p_2) \geq 3$.

Now, we address the second part of the statement. The same proof as in Proposition \ref{prop: every point on multisecant} (for $r=2$) shows that \eqref{(ii)} implies \eqref{(i)}. 

In order to prove that \eqref{(i)} implies \eqref{(ii)}, we first show that $\dim \bbP L_1 = \dim \bbP L_2 = 1$. Clearly $\dim \bbP L_1,\dim \bbP L_2\in \{1,2\}$, because if $\dim \bbP L_i = 0$, $i=1,2$, then $p_i \in X_i$. Similarly to the first part of the proof, we reduce the problem to the span of rank one matrices:

\begin{itemize}
 \item Let $\dim \bbP L_1 = \dim \bbP L_2 = 2$. After a suitable choice of coordinates, 
 \[
   a_1 \otimes b_1 = \left( \begin{smallmatrix}
                          1 & 0 & 0 \\ 0 & 0 & 0\\ 0 & 0 & 0
                         \end{smallmatrix} \right), \qquad
   a_2 \otimes b_2 = \left( \begin{smallmatrix}
                          0 & 0 & 0 \\ 0 & 1 & 0\\ 0 & 0 & 0
                         \end{smallmatrix} \right), \qquad
   a_3 \otimes b_3 = \left( \begin{smallmatrix}
                          0 & 0 & 0 \\ 0 & 0 & 0\\ 0 & 0 & 1
                         \end{smallmatrix} \right) ,
 \]
and their span does not contain other rank one matrices, providing a contradiction.
\item Let $\dim \bbP L_1 = 2$ and $\dim \bbP L_2 = 1$. At least two among $b_1,b_2,b_3$ are distinct, so we may assume $b_3 \in \langle b_1 , b_2\rangle$. Passing to the affine cones, suppose $b_3 = \lambda_1 b_1 + \lambda_2 b_2$. After a suitable choice of coordinates, 
\[
    a_1 \otimes b_1 = \left( \begin{smallmatrix}
                          1 & 0 \\ 0 & 0 \\ 0 & 0 
                         \end{smallmatrix} \right), \qquad
   a_2 \otimes b_2 = \left( \begin{smallmatrix}
                          0 & 0 \\ 0 & 1\\ 0 & 0 
                         \end{smallmatrix} \right), \qquad
   a_3 \otimes b_3 = a_3 \otimes (\lambda_1 b_1 + \lambda_2 b_2) = \left( \begin{smallmatrix}
                          0 & 0  \\ 0 & 0 \\ \lambda_1 & \lambda_2
                         \end{smallmatrix} \right).
\]
If $\lambda_1,\lambda_2$ are both nonzero, then their span does not contain other rank one matrices; contradiction. If $\lambda_1 = 0$, then $b_3 = b_2$, and $p_1 \otimes p_2 \in \langle a_1 \otimes b_1, a_2 \otimes b_2, a_3 \otimes b_2\rangle$, so that $p_1 \otimes p_2 \in \langle a_1 \otimes b_1, \tilde{a} \otimes b_2 \rangle$, for some $\tilde{a} \in \langle a_2,a_3\rangle$. Since $a_1,\tilde{a}$ are linearly independent and $b_1,b_2$ are linearly independent, we obtain a contradiction as in the previous case; analogously for $\lambda_2 = 0$.
\item Let $\dim \bbP L_1 = 1$ and $\dim\bbP L_2 = 2$. This is analogous to the previous case.
\end{itemize}

Therefore, we are left with the case $\dim \bbP L_1 = \dim \bbP L_2 = 1$. 

\begin{quote}
{\bf Claim.} \textit{$\bbP L_1$ and $\bbP L_2$ are multisecant lines to $X_1$ and $X_2$, respectively.}

\begin{proof}[Proof of Claim] We are going to show that $\sharp\{a_1,a_2,a_3\}=\sharp\{b_1,b_2,b_3\}=3$. Suppose, by contradiction, that the $a_i$'s are not distinct, and assume $a_2=a_3$. There are two cases:
\begin{itemize}
\item if $b_1,b_2,b_3$ are distinct, then $\bbP L_2$ is a multisecant line to $X_2$ and we may assume $b_3 = \lambda_1 b_1 + \lambda_2 b_2$. In this case, after a suitable choice of coordinates,
\[
 a_1 \otimes b_1 = \left( \begin{smallmatrix}
                          1 & 0 \\ 0 & 0
                         \end{smallmatrix} \right), \quad
a_2 \otimes b_2 = \left( \begin{smallmatrix}
                          0 & 0 \\ 0 & 1
                         \end{smallmatrix}\right), \quad
a_3 \otimes b_3 = a_2 \otimes(\lambda_1 b_1 + \lambda_2 b_2) = \left( \begin{smallmatrix}
                          0 & 0 \\ \lambda_1 & \lambda_2
                         \end{smallmatrix} \right),
\]
the only other rank one matrices in their span are of the form $\left( \begin{smallmatrix}
                          \mu_1 & 0 \\ \mu_2 & 0
                         \end{smallmatrix} \right) = (\mu_1 a_1 + \mu_2 a_2) \otimes b_1$ and this would imply $p_2 = b_1$, in contradiction with the hypothesis;
\item if $b_1,b_2,b_3$ are not distinct, then there are two possibilities: either $b_1=b_2$ (equivalently $b_1 = b_3$) or $b_2 = b_3$. Again, after a suitable choice of coordinates, we have one of the following possibilities:
\begin{itemize}
\item if $b_1=b_2$, then 
\[
 a_1 \otimes b_1 = \left( \begin{smallmatrix}
                          1 & 0 \\ 0 & 0
                         \end{smallmatrix} \right), \quad
a_2 \otimes b_2 = a_2 \otimes b_1 = \left( \begin{smallmatrix}
                          0 & 0 \\ 1 & 0
                         \end{smallmatrix}\right) , \quad
a_3 \otimes b_3 = a_2 \otimes b_3 = \left( \begin{smallmatrix}
                          0 & 0 \\ 0 & 1 
                         \end{smallmatrix} \right):
\]
the only other rank one matrices in their span are of the form $\left( \begin{smallmatrix}
                          \mu_1 & 0 \\ \mu_2 & 0
                         \end{smallmatrix} \right) =  (\mu_1 a_1 + \mu_2 a_2) \otimes b_1$ or $\left( \begin{smallmatrix}
                           0 & 0 \\ \mu_1 & \mu_2
                         \end{smallmatrix} \right) =  a_2 \otimes ( \mu_1 b_1 + \mu_2 b_3)$. In both cases we obtain a contradiction.
\item if $b_2 = b_3$, then $a_2 \otimes b_2 = a_3 \otimes b_3$ and we would obtain $\rmR_{X_1 \times X_2}(p_1 \otimes p_2) = 2$, in contradiction with the first part of the proof.
                         \end{itemize}
\end{itemize}\vspace{-0.7cm}
\end{proof}
\end{quote}
Therefore, we proved that $\bbP L_1$ (resp. $\bbP L_2$) is a multisecant line to $X_1$  (resp. $X_2$)  and its intersection with $X$ contains the set of points  $\{a_1,a_2,a_3\}$ (resp. the set of points $\{b_1,b_2,b_3\}$).
\\ Let $\phi: \bbP L_1 \to \bbP L_2$ be the unique linear map such that $\phi(a_j) = b_j$. Consider the identification $\phi \times {\rm id}_{\bbP L_2} : \bbP L_1 \times \bbP L_2 \to \bbP L_2 \times \bbP L_2$. By linearity, we have that $\phi(p_1) \otimes p_2 \in \langle \phi(a_1) \otimes b_1 , \phi(a_2) \otimes b_2 , \phi(a_3) \otimes b_3 \rangle = \langle b_1^{\otimes 2} , b_2 ^{\otimes 2} , b_3^{\otimes 2}\rangle$. This shows that $\phi(p_1) \otimes p_2$ is a symmetric rank one element of $\bbP(V_2\otimes V_2)$, namely it belongs to $\nu_2(\bbP L_2)$, so that $p_2 = \phi(p_1)$. This concludes the proof.
\end{proof}
As a direct corollary, we have the following classification of the $X^{\times 2}$-ranks attained by $p^{\otimes 2}$ for points $p \in \bbP V$ with $X$-rank equal to $2$: this is obtained by combining Theorem \ref{thm: rank 2 x rank 2 classification} and Proposition \ref{prop: every point on multisecant}.
\begin{corollary}\label{corol: trisec with one point}
Let $X \subset \bbP V$ be a variety. Let $p \in \bbP V$ with $\rmR_X(p) = 2$. Then 
\[
3 \leq \rmR_{X \times X}\left(p^{\otimes 2}\right) \leq 4,
\]
and $\rmR_{X \times X}\left(p^{\otimes 2}\right)= 3$ if and only if $p$ lies on a multisecant line to $X$.
\end{corollary}

\subsection{Examples of trisecant lines to space curves} The results we mention here are classical: we include them with a particular focus on characterizing points lying on generalized trisecant lines, namely those lines whose intersection with $X$ consists of a $0$-dimensional scheme of degree at least~$3$. A nondegenerate smooth curve $X \subset \bbP^3$ has infinitely many trisecant lines (their closure is a surface), unless the degree $d$ and genus $g$ of $X$ are $(d,g) = (3,0),(4,1)$; see \cite[Proposition~1, Remark~1]{Ber:SingularitiesTrisecantSurface}. When $X\subset \bbP^3$ is smooth (and in characteristic zero), the generic trisecant line has indeed three distinct points of intersection with the curve; this is a consequence of the results of \cite{k}.

\begin{example}[Twisted cubic] If $(d,g) = (3,0)$, then $X$ is the twisted cubic, which has no trisecant line. Indeed, the ideal of the twisted cubic is generated by quadrics; hence, a trisecant line would have to be contained in every quadric surface containing $X$ and, in particular, in $X$. By Corollary~\ref{corol: trisec with one point}, we recover a multiplicativity result for points with $X$-rank equal to $2$ in the case of the twisted cubics: if $p \in \bbP^3$ satisfies $\rmR_X(p) = 2$, then $\rmR_{X\times X}(p^{\otimes 2}) = 4$. 
\end{example}

\begin{example}[Elliptic normal quartic] \label{example: elliptic normal quartic}
If $(d,g) = (4,1)$, then $X$ is an elliptic normal curve. The Riemann-Roch Theorem \cite[Theorem IV.1.3]{h} provides $h^0(\calO_X(2)) = 8$. To see this, let $K_X$ be the canonical divisor of $X$; recall $\deg(K_X) = 2g-2= 0$. Thus the Riemann-Roch Theorem, applied to the divisor $D = 2H$ where $H$ is the hyperplane section, gives $h^0( \calO_X(2)) = \deg( \calO_X(2) )  + g - 1 = 8$. The long exact sequence in cohomology associated to the short exact sequence $0 \to \calI_X(2) \to \calO_{\bbP^3}(2) \to \calO_X(2) \to 0$ gives $h^0(\calI_X(2)) \geq 2$. Thus $X$ is contained in the intersection of two quadric surfaces $Q_1,Q_2$. Since $\deg(X) = 4 = \deg( Q_1 \cap Q_2)$, $X$ is exactly the intersection of two quadrics. Every trisecant line to $X$ is contained in every quadric containing $X$, so $X$ has no trisecant lines. As above, if $p \in \bbP^3$ satisfies $\rmR_X(p) = 2$, then $\rmR_{X\times X}(p^{\otimes 2}) = 4$. 
\end{example}

\begin{example}[Rational quartic in $\bbP^3$] \label{subsec: rational quartic}
If $(d,g) = (4,0)$, then $X$ is a linear projection of the rational normal curve in $\bbP^4$. As before, by Riemann-Roch, $h^0(\calO _X(2)) =9$. So the long exact sequence associated to $0 \to \calI_X(2) \to \calO_{\bbP^3}(2) \to \calO_X(2) \to 0$ gives $h^0(\calI_X(2)) \geq 1$. 

Then $X$ is contained in at least one quadric surface $Q$. Recall that $Q$ is unique, irreducible and smooth: 

\begin{itemize}
 \item \underline{Uniqueness}. If $Q$ is not unique, then $X$ would be a complete intersection of two quadrics and therefore would be of genus $g =1$. 
 \item \underline{Irreducibility}. If $Q$ is reducible, then it is union of two planes, in contradiction with the fact that $X$ is non-degenerate. 
 \item \underline{Smoothness}. This is more involved and follows from \cite[Exercise V.2.9]{h}. 
  \end{itemize}

Notice that if $\bbP L$ is a trisecant line to $X$, then $\bbP L \subseteq Q$, therefore if $p \notin Q$ and $\rmR_X(p) = 2$ then $p$ does not lie on a trisecant line and $\rmR_{X \times X}(p^{\otimes 2}) =4$.

We analyze points in $Q$. Since $Q$ is smooth, we have $Q \simeq \bbP^1 \times \bbP^1$ and $X$ is a divisor on $Q$. Since $\deg(X) = 4$, $X$ is a divisor of bidegree $(a,b)$ with $a+b = 4$: clearly $a,b > 0$ because $X$ is non-degenerate. If $(a,b) = (2,2)$, $X$ is the elliptic quartic of Example \ref{example: elliptic normal quartic}, in contradiction with the rationality of $X$. Therefore $X$ is a divisor of bidegree $(3,1)$ or $(1,3)$. Assume $X \in | \calO_Q(3,1)|$. Now, lines in $Q$ coincide with the lines of the two rulings. If $\bbP L \in |\calO_Q(0,1)|$, then $\deg(\bbP L \cap X) = 0 \cdot 3 + 1 \cdot 1 = 1$, so $\bbP L$ cannot be a trisecant line. If $\bbP L  \in |\calO_Q(1,0)|$, then $\deg(\bbP L \cap X) = 1 \cdot 3 + 0 \cdot 1 = 3$, so that $\bbP L \cap X$ is a $0$-dimensional scheme of degree $3$, and $\bbP L$ is a trisecant line if and only if $\bbP L \cap X$ is reduced. Theorem \ref{thm: rank 2 x rank 2 classification} guarantees that the points lying on these lines are the unique elements of $X$-rank $2$ for which strict submultiplicativity of rank occurs. 

Let $\pi : X \to \bbP^1$ be the restriction of the projection of $Q \simeq \bbP^1 \times \bbP^1 \to \bbP^1$ on the first factor. Then $\pi$ is a finite morphism of degree $3$; its fibers are generically reduced, showing that every line $\bbP L \in |\calO_Q(1,0)|$ is a trisecant (with $L\cap X$ reduced) except at most finitely many of them. The lines in $|\calO_Q(1,0)|$ for which $\bbP L \cap X$ is not reduced correspond to the ramification locus $\calR$ of $\pi$. By Hurwitz Formula \cite[Corollary 2.4]{h}, the degree of the ramification locus is $\deg(\calR) = 2g - 2 - 2dg + 2d = 4$. This shows that there are at most $4$ lines $\bbP L \in |\calO_Q(1,0)|$ such that $\bbP L \cap X$ is not-reduced: equality holds if and only if the ramification locus is reduced; a more delicate argument, shows that indeed, there are always at least two distinct lines such that $\bbP L\cap X$ is not reduced. 

Therefore, strict submultiplicaitivity of a point $p$ having $X$-rank $2$, occurs if and only if $p\in \bbP L \in |\calO_Q(1,0)|$ with $\bbP L \notin \calR$.
\end{example}

\begin{remark}
For curves $X \subseteq \bbP^4$, one expects a finite number of trisecant lines. This number was determined by Castelnuovo and Berzolari, see \cite[p. 435]{lb9}. If $X \subseteq \bbP^4$ has degree $d$ and genus $g$, then the expected number of trisecant lines is
\[
\frac{(d-2)(d-3)(d-4)}{6} - g(d-4) .
\]
If $X \subseteq \bbP^N$ for $N \geq 5$ then one expects $X$ to have no trisecant lines, so Theorem \ref{thm: rank 2 x rank 2 classification} implies that the rank multiplicativity holds for points of $X$-rank $2$.
\end{remark}

The idea of having a trisecant line to guarantee the strict rank multiplicativity of the $X$-rank 2 points can be extended to the existence of multisecant spaces of higher dimension. 

\subsection{Multisecant $r$-dimensional planes}
\begin{definition}
	An $r$-dimensional linear space $\bbP W \simeq \bbP^r$ is \textbf{multisecant} to a variety $X \subset \bbP V$ if the intersection $\bbP W \cap X$ contains at least a set of $r+2$ distinct points. 
\end{definition}

\begin{proposition}\label{prop: Pr r+2 secant}
Let $X \subseteq \bbP V$ be a variety and $\bbP W \simeq \bbP ^r \subseteq \bbP V$ be a multisecant linear space to~$X$. Suppose that $\bbP W $  does not contain a multisecant $\bbP^s$, for any $s<r$. For every $p \in \bbP W$, we have 
\[
 \rmR_{X^{\times (r+1)}} \left(p^{\otimes (r+1)}\right) \leq (r+1)^{r+1} - (r+1)! + 1.
\]
\end{proposition}
\begin{proof}
Let $z_0 \vvirg z_r,w \in \bbP W\cap X$ so that $w \in \langle z_0 \vvirg z_r \rangle = \bbP W$. Since $\bbP W$ does not contain any smaller multisecant space, the point $w$ is not in the linear span of any proper subset of the $z_i$'s. For every sequence of $(r+1)$ non-negative integers $\alpha = (\alpha_0 \vvirg \alpha_r )$ with $\alpha_j \leq r$, write $\bfz_\alpha = z_{\alpha_0} \mathcal \otimes \cdots \otimes z_{\alpha_r}$. By definition,
 \[
  p^{\otimes (r+1)} \subseteq \langle \bfz_{\alpha} : \alpha \in \lbrace 0\vvirg r\rbrace^{r+1} \rangle = \bbP  W^{\otimes (r+1)} \subseteq \bbP V^{\otimes (r+1)}.
 \]
 Let $\Sigma = \{\bfz_\sigma ~|~ \sigma \in \frakS_{r+1} \text{ is a permutation of } \{0,\ldots,r\}\}$. Then, let
$$
	\bbP E = \langle \bfz_\alpha ~|~ \bfz_\alpha \in \Sigma \rangle \quad \text{ and } \quad \bbP F = \langle \bfz_\alpha ~|~ \bfz_\alpha \not\in \Sigma \rangle.
$$
Clearly $\dim E = (r+1)!$, $\dim F = (r+1)^{r+1} - (r+1)!$ and both $E$ and $F$ are invariant under the action of $\frakS_{r+1}$ which permutes the factors. 
\begin{quote}
{\bf Claim. } \textit{$p^{\otimes (r+1)} \in \langle \bbP F \cup w^{\otimes (r+1)}\rangle$. }
\begin{proof}[Proof of Claim.]
If $p^{\otimes (r+1)} \in \bbP F$, then $p^{\otimes (r+1)} \in \langle \bbP F \cup w^{\otimes (r+1)}\rangle$.

Assume $p^{\otimes (r+1)} \notin \bbP F$. Consider the linear projection $\pi_{F}:\bbP W^{\otimes (r+1)} \to \bbP E$ from $\bbP F$, defined in terms of the chosen basis. Remark that $\pi_{F}$ is defined on  $p^{\otimes (r+1)}$ since $p^{\otimes (r+1)} \not\in \bbP F$. Moreover, $\pi_F$ is $\frakS_{r+1}$-equivariant and, since $p^{\otimes (r+1)} \in \bbP W^{\otimes (r+1)}$ is symmetric, $\pi_F\left(p^{\otimes (r+1)}\right)$ is symmetric as well. The only symmetric element in $\bbP E$ is $e = \sum_{\sigma \in \frakS_{r+1}} \bfz_\sigma$ and, therefore, we have that $\pi_F\left(p^{\otimes (r+1)}\right) =e$. Since $w$ is a non-trivial linear combination of all the $z_i$'s, we deduce that $w^{\otimes (r+1)} \notin \bbP F$ and, in particular, $\pi_F$ is defined on $w^{\otimes (r+1)}$ so that $\pi_F\left(w^{\otimes (r+1)} \right)= e$ as well. Since $p^{\otimes (r+1)}$ and $w^{\otimes (r+1)}$ have the same image under $\pi_F$, the line $\langle w^{\otimes (r+1)}, p^{\otimes (r+1)} \rangle$ intersects $\bbP F$, showing that $p^{\otimes (r+1)} \in \langle \bbP F \cup w^{\otimes (r+1)} \rangle$.
\end{proof}\end{quote}
Since $\bbP F$ is spanned by rank-one elements with respect to $X^{\times (r+1)}$, from the Claim we get that 
$$
 \rmR_{X^{\times (r+1)}} \left(p^{\otimes (r+1)}\right) \leq \dim F + 1 = (r+1)^{r+1} - (r+1)! + 1.\vspace{-0,7cm}
$$ 
\end{proof}

 \section{Sufficient conditions for strict submultiplicativity of border rank}\label{sec:geometric_cond}
The existence of multisecant lines turns out to be a key tool for the strict submultiplicativity of the $X$-rank. The geometric concept of multisecant line can be generalized to what we will call \emph{$r$-multidrop lines}; they turn out to be a valuable tool to verify the strict submultiplicativity of the border $X$-rank. In this section we study sufficient conditions that guarantee the existence of such $r$-multidrop lines that we introduce next. 
 
 \subsection{The multidrop construction} 
The multidrop construction was firstly introduced in \cite[Section 4]{ChrGesJen:BorderRankNonMult}: 

\begin{definition}
Let $X \subseteq \bbP V$ be a variety and let  $z,q_0,q_1$ three points on a line $\bbP L\subseteq \bbP V$ such that
\begin{itemize}
 \item $z \in \bbP L \cap X$;
 \item $q_0,q_1 \in \bbP L \cap \sigma_r(X)$, with $q_0,q_1,z$ distinct;
 \item $\bbP L \not\subseteq \sigma_r(X)$.
\end{itemize}
The line $\bbP L$ is said to be an \textbf{$r$-multidrop line} for $X$.
\end{definition}

Notice that in this case we have $\bbP L \subseteq \sigma_{r+1}(X)$. The existence of such a line guarantees the existence of points realizing strict submultiplicativity of border $X$-rank; for this reason, we call such a line an $r$-multidrop line for $X$. Note that $1$-multidrop lines are simply multisecant lines in the sense of Definition \ref{def: multisecant line}.

Choose coordinates on $\bbP L \simeq \bbP^1$ such that $z = (1:0)$, $q_0 = (0:1)$ and $q_1 = (1:1)$. The line $\bbP L$ is parametrized by 
\[
\begin{matrix}
	\ell: & \bbP^1 & \to & \bbP V \\ 
	& (\eps_0:\eps_1) &\mapsto & \ell(\eps_0:\eps_1) = \eps_0z + \eps_1q_0,
\end{matrix}
\]
so that $\ell(1:1) = q_1$. Let $p =  q_0 + 2z = q_1 +z$. Points constructed in this way are used in \cite{ChrGesJen:BorderRankNonMult} to obtain examples of strict submultiplicativity of border rank.

The following lemma was proved in \cite{ChrGesJen:BorderRankNonMult} for  a single variety $X$ and strict submultiplicativity for $p^{\otimes r}$ was considered. We state it more generally for the case of two distinct varieties; the proof is essentially the same.

\begin{lemma}[{\cite[Lemma 4.1]{ChrGesJen:BorderRankNonMult}}]\label{lemma: multidrop}
Let $X_1 \subset \bbP V_1$ and let $X_2 \subset \bbP V_2$. Let $\bbP L_1$ and $\bbP L_2$ such that $z_i \in \bbP L_i \cap X_i$ and $\{q_{i,0},q_{i,1}\} \subseteq \bbP L_i \cap \sigma_{r_i}(X)$, for $i = 1,2$, namely $\bbP L_1,\bbP L_2$ are $r_1$- and $r_2$-multidrop lines for $X_1,X_2$, respectively. Let $p_i = q_{i,0}+2z_i$, for $i = 1,2$. Then,
\[
\brank_{X_1\times X_2}(p_1\otimes p_2) \leq r_1r_2 + r_1 + r_2 < (r_1+1)(r_2+1).
\]
In particular, if $\uR_{X_i}(p_i) = r_i+1$, we obtain $\uR_{X_1 \times X_2}(p_1 \otimes p_2) < \uR_{X_1}(p_1) \cdot \uR_{X_2}(p_2)$.
\end{lemma}

The existence of multidrop lines is a strong condition and determining whether a variety admits them is not easy.

In \cite{ChrGesJen:BorderRankNonMult}, the case where $\sigma_r(X)$ is a hypersurface was considered. In this case, the existence of $r$-multidrop lines can be determined by studying the multiplicity of a point of $X$ in $\sigma_r(X)$. Briefly, suppose $\sigma_r(X)$ is a hypersurface, fix $z \in X$ and let $\bbP L$ be a generic line through $z$. By Bezout's Theorem, $\bbP L \cap \sigma_r(X)$ consists of $\deg(\sigma_r(X))$ points counted with multiplicity and, by Bertini's Theorem, all intersection points except $z$ are non-singular points of $\sigma_r(X)$. More precisely, these are $\deg(\sigma_r(X)) - \mult_{\sigma_r(X)}(z)$ points, where $\mult_{\sigma_r(X)}(z)$ is the multiplicity of $z$ in $\sigma_r(X)$. In particular, if $\deg(\sigma_r(X)) - \mult_{\sigma_r(X)}(z) \geq 2$, then $L$ is a $r$-multidrop line and it is possible to apply Lemma \ref{lemma: multidrop}.

In this section, we investigate conditions to guarantee that $\deg(\sigma_r(X)) - \mult_{\sigma_r(X)}(z) \geq 2$ in the hypersurface case.

First, notice that if $\deg(\sigma_r(X)) - \mult_{\sigma_r(X)}(z) = 0$, then every line through $z$ only intersects $\sigma_r(X)$ at $z$; this implies that a line through $z$ and $q \in \sigma_r(X)$, for $q\in \sigma_r(X)$, must be entirely contained in $\sigma_r(X)$. This shows that $\sigma_r(X)$ is a cone with vertex containing $z$. In particular, we may assume that there exists $z \in X$ with $\deg(\sigma_r(X)) - \mult_{\sigma_r(X)}(z) \geq 1$, because if this is not the case, then $\sigma_r(X)$ is a cone over every point of $X$, against the non-degeneracy condition of $X$.

\subsection{Rationality}

First observe that if $\deg(\sigma_r(X)) - \mult_{\sigma_r(X)}(z) =1$, for some $z \in X$, then $\sigma_r(X)$ is a rational variety.

\begin{proposition}\label{prop: gap one implies birational}
 Let $X \subseteq \bbP^N$ be a variety. Suppose that $\sigma_r(X)$ is a hypersurface and let $z \in X$ be a point with $\deg(\sigma_r(X)) - \mult_{\sigma_r(X)}(z) =1$. Then, $\sigma_r(X)$ is rational.
\end{proposition}
\begin{proof}
 Let $\pi_z : \sigma_r(X) \dashto \bbP^{N-1}$ be the projection from $z$. The map $\pi_z$ is regular on $\sigma_r(X)\setminus\{z\}$. If $\bbP L$ is a generic line through $z$, then $\sigma_r(X) \cap \bbP L = \{ z, q\}$, for a point $q \in \sigma_r(X)$. In particular, $q$ is the only point in $\pi_z^{-1}(\pi_z(q))$, so $\pi_z$ is generically one-to-one and, therefore, birational.
\end{proof}

Proposition \ref{prop: gap one implies birational} shows that if $\sigma_r(X)$ is not rational, then generic lines intersecting $X$ are multidrop lines. We provide a technical condition which guarantees that $\sigma_r(X)$ is not rational. 

For a set $U$ and a positive integer $m$, let $\Sigma^m U := U^{\times m} / \frakS_m$ be the symmetric product of $m$ copies of $U$, that is the quotient of $U^{\times m}$ under the action of the symmetric group $\frakS_m$ permuting the factors. If $X$ is a projective algebraic variety, then $\Sigma^m X$ is a projective algebraic variety, as well \cite[Lecture 10]{Harris:AlgGeo}. The $r$-th \emph{abstract secant variety} of $X \subseteq \bbP^N$ is
\[
 \calS^\circ_r(X) := \{ ( (x_1 \vvirg x_r), p) \in \Sigma^r X \times \bbP^N : p \in \langle x_1 \vvirg x_r \rangle \} \subseteq \Sigma^r X \times \bbP^N.
\]
Then $ \calS^\circ_r(X)$ is an irreducible locally closed set and the projection to $\bbP^N$ surjects onto the set of points having $X$-rank at most $r$. Write $\calS_r(X) = \bar{\calS^\circ_r(X)}$. 

A variety $X$ is \emph{generically $r$-identifiable} if the projection of $\calS_r(X)$ to $\bbP^N$ is generically one-to-one, hence birational. In particular, if $X$ is generically $r$-identifiable, then $\sigma_r(X)$ is rational if and only if $\calS_r(X)$ is rational. We refer to \cite{BalBerCh:DimensionContactLoci} and the references therein for a complete explanation of the notion of identifiability and related topics.

\begin{proposition}\label{prop: map to abelian implies not rational}
 Let $X$ be generically $r$-identifiable and suppose there exists a non-constant map $\alpha : X \to A$ where $A$ is an abelian variety. Then, $\sigma_r(X)$ is not rational.
\end{proposition}
\begin{proof}
 It is a known fact that if a variety $Y$ is rational and $A$ is an abelian variety, then every map $\beta : Y \to A$ is constant, see, e.g., \cite[Proposition 3.9]{Milne:AbelianVarieties}. Define $\alpha_\Sigma : \Sigma^r X \to A$ by $\alpha_{\Sigma} ( x_1 \vvirg x_r) = \alpha(x_1) + \cdots + \alpha(x_r)$, where $+$ denotes the (abelian) operation of $A$ as a group; since $+$ is commutative, $\alpha_\Sigma$ is well-defined. Notice that since $\alpha$ is not constant, then $\alpha_{\Sigma}$ is non-constant.
 
In particular, the map $\calS_r(X) \to \Sigma^r X \to A$ given by the composition of the projection of $\calS_r(X)$ onto $\Sigma^r X$ followed by $\alpha_\Sigma$ is non-constant as well. This shows that $\calS_r(X)$ is not rational, and therefore, since by generic identifiability $\calS_r(X)$ is birational to $\sigma_r(X)$, we conclude that $\sigma_r(X)$ is not rational.
\end{proof}

We mention that, given a smooth projective variety $X$, there is a general construction to define an abelian variety $\Alb(X)$, called {\it Albanese variety}, together with a map $\alpha : X \to \Alb(X)$ satisfying a universal property, see \cite[Theorem V.13]{beauville}. We do not provide any detail here, but we point out that one of the properties of this construction is that $\alpha(X)$ generates $\Alb(X)$ as an abelian group and that $\dim \Alb(X) = h^1(\calO_X)$. In particular, if $h^1(\calO_X) > 0$, then $\alpha: X \to \Alb(X)$ is non-constant. Hence, Proposition \ref{prop: map to abelian implies not rational} provides immediately the following.
\begin{corollary}\label{corol: positive h1 implies not rational}
 Let $X$ be a smooth projective variety, generically $r$-identifiable and such that $h^1(\calO_X) > 0$. Then, $\sigma_r(X)$ is not rational. Moreover, if $\sigma_r(X)$ is a hypersurface, then a generic line through $X$ is a $r$-multidrop line.
\end{corollary}

We conclude this section providing a class of varieties to which Corollary \ref{corol: positive h1 implies not rational} can be applied.

\begin{example}\label{example: smooth nonrational curve}
Let $X \subset \mathbb P^{2k}$ be a smooth curve of genus $g>0$. By Palatini's Lemma (see, e.g., \cite[Corollary 1.2.3]{Russo:TangentsSecants}), its $k$-th secant variety is a hypersurface. Generic identifiability always holds for curves, see \cite[Corollary 2.7]{cc2}. The condition $h^1(\calO_X) > 0$ follows by Serre's duality, since $h^1(\calO_X) = h^0(\omega_X) = g$, where $\omega_X$ is the canonical bundle on $X$. This shows that $X$ satisfies the hypotheses of Corollary \ref{corol: positive h1 implies not rational} and a generic line through $X$ is a $r$-multidrop line.
\end{example}

Example \ref{example: smooth nonrational curve} generalizes the example of \cite[Section 5.2]{ChrGesJen:BorderRankNonMult} where it was shown that a generic line through the elliptic normal quintic $X \subseteq \bbP^4$ is a $2$-multidrop. In that case, the result was shown directly using the equation of the hypersurface $\sigma_2(X)$ which provides $\deg(\sigma_2(X)) - \mult_{\sigma_2(X)}(z) = 5 -3 = 2$. In fact, the result of the next subsection give a full generalization of Example \ref{example: smooth nonrational curve}.

\subsection{The case of curves}

We completely characterize the existence of multidrop lines for curves admitting a secant variety that is a hypersurface. Briefly, if $X \subseteq \bbP^{2k}$ is a non-degenerate curve, then $X$ always admits $k$-multidrop lines, except if $X$ is the rational normal curve in $\bbP^{2k}$. 

\begin{theorem}\label{thm: curves}
 Let $X \subseteq \bbP^{2k}$ be a non-degenerate curve and let $z \in X$ be a generic point. Then $\sigma_k(X)$ is a hypersurface. Moreover, $\deg(\sigma_k(X)) - \mult_{\sigma_k(X)}(z) =1$ if and only if $X$ is a rational normal curve of degree $2k$.
\end{theorem}
\begin{proof}
 By Palatini's Lemma (see e.g. \cite[Proposition 1.2.3]{Russo:TangentsSecants}), $\dim \sigma_k(X) = 2k-1$, so it is a hypersurface. Let $z \in X$ be a generic point and let $\mu = \mult_{\sigma_k(X)}(z)$. Let $\pi_z : \bbP^{2k} \setminus \{ z \} \to \bbP^{2k-1}$ be the projection from $z$ and let $X' = \bar{\pi_z(X)} \subseteq \bbP^{2k-1}$. Since $z$ is generic, the projection $\pi_z$ is a birational map between $X$ and $X'$; in particular $\pi_z$ is generically one-to-one on $X$. Moreover, $X'$ is a non-degenerate curve and $\sigma_k(X') = \bbP^{2k-1}$.
 
 Let $\bbP L$ be a generic line through $z$. Let $q' = \pi_z(\bbP L)$, which is a generic point of $\bbP^{2k-1}$. Let $u: \calS_k(X') \to \bbP^{2k-1}$ be the projection from the abstract secant variety of $X'$ to $\bbP^{2k-1}$. Since $\dim \calS_k(X') = 2k-1$, we deduce that $u$ is surjective and generically finite.
 
 Since $q'$ is generic in $\bbP^{2k-1}$, $\deg(u)$ can be computed as the number of preimages of $q'$ via $u$. Every element of $u^{-1}(q')$ is a pair $(S',q') \in \Sigma^kX' \times \bbP^{2k-1}$ where  $S' \subseteq X'$ with $|S'| = k$ and $q' \in \langle S' \rangle$. Now, from each element $(S',q') \in u^{-1}(q')$, we construct different points $q_{S'} \in \sigma_r(X) \cap \bbP L$, such that all these points are also distinct from $z$. In particular, we conclude $\deg(\sigma_k(X)) - \mu \geq \deg(u)$.
 
 Since $\pi_z$ is generically one-to-one on $X$ and $q'$ is generic in $\bbP^{2k-1}$, for every $(S',q') \in u^{-1}(q')$ and every $w' \in S'$, there is a unique $w \in X$ such that $\pi_z(w) = w'$. In particular, every $(S',q') \in  u^{-1}(q')$ defines a subset $S \subseteq X$ with $|S| = k$. Notice that for all $S$ constructed in this way, $z \notin \langle S \rangle$`: indeed, $\langle S \rangle$ is a generic $\bbP^{k-1}$ secant to $X$ and therefore by the Trisecant Lemma (see e.g. \cite[Proposition 1.4.3]{Russo:TangentsSecants}), $\langle S \rangle$ only intersects $X$ at the points of $S$. For every $S$, let $q_{S'}$ be the unique preimage of $q'$ in $\langle S \rangle$. Notice that $q_{S'}$ can only arise from a single $(S',q') \in u^{-1}(q')$. Indeed, $X$ is generically $r$-identifiable (see e.g. \cite[Corollary 2.7]{cc1}) and $q_{S'}$ is a generic point of $\sigma_k(X)$, since it arises as intersection of $\sigma_k(X)$ with a generic line through $z$. In particular, for each $S'$, the preimage $S$ is the only subset of $X$ with $|S|=k$ and $q_{S'} \in \langle S \rangle$. This shows $\deg(\sigma_k(X)) - \mu \geq \deg(u)$.
 
Therefore, we conclude that $\deg(\sigma_k(X)) - \mult_{\sigma_k(X)}(z) \geq 2$ whenever $\deg(u) \geq 2$. By \cite[Theorem 3.4]{CatJ:HomIdHigherSecVar}, the only curve $X' \subseteq \bbP^{2k-1}$ with $\deg(u) = 1$ is a rational normal curve of degree $2k-1$. If $X'$ is a rational normal curve in $\bbP^{2k-1}$ then $X$ is a rational normal curve in $\bbP^{2k}$, showing that $\deg(\sigma_k(X)) - \mult_{\sigma_k(X)}(z) \geq 2$ if $X$ is not the rational normal curve.

Conversely, if $X$ is the rational normal curve in $\bbP^{2k}$, then $\sigma_k(X)$ is the hypersurface of degree $k+1$ defined by the determinant of a square matrix of size $k+1$, see \cite[Proposition 9.7]{Harris:AlgGeo}. In this case, one can observe $\mult_{\sigma_r(X)}(z) = k-1$ and therefore $\deg(\sigma_k(X)) - \mult_{\sigma_k(X)}(z) \geq 1$.
\end{proof}

Theorem \ref{thm: curves} generalizes Example \ref{example: smooth nonrational curve} as well as the examples on curves of \cite[Section 5.2 and Section 5.3]{ChrGesJen:BorderRankNonMult}.

\subsection{Multidrop lines via projection}\label{subsec: secants via projections}
We saw that the existence of multidrop lines implies the existence of points for which border rank strict submultiplicativity occurs. The construction of the previous sections relies on the fact that one of the secant varieties is a hypersurface. In this section, we provide a more general construction to generate multidrop lines which does not depend on the dimension of the secant variety. 

Let $X \subseteq \bbP ^N$ and fix $r$ such that $\sigma_r(X) \neq \bbP ^N$. Let $z \in X$ and $q_0,q_1 \in \sigma_r(X)$ be generic points: in particular $z,q_0,q_1$ are not collinear and $\uR_X(q_0) = \rmR_X(q_0) = \uR_X(q_1) = \rmR_X(q_1) = r$.
A generic $w \in \langle q_0,q_1,z\rangle$ is such that $\uR_X(w) = \rmR_X(w) = \min \{ 2r+1 ,g_X\}$, where $g_X$ is the generic $X$-rank in 
$\bbP^N$. 

Let $\pi_w : \bbP^N \dashto \bbP^{N-1}$ be the projection from $w$. Since $w \notin \sigma_r(X)$, $\pi_w$ is regular on $\sigma_r(X)$ and by genericity $\uR_{\pi_w( X)} (\pi_w(q_0)) = \uR_{\pi_w(X)} ( \pi_w(q_1)) = r$ and similarly for the $\pi_w(X)$-rank.

\begin{proposition}\label{prop: projection generates multidrop}
 In the construction above, let $\bbP L = \langle \pi_w(z) ,\pi_w(q_0) \rangle$. If $\bbP L \not\subseteq \sigma_r(\pi_w(X))$, then $\bbP L$ is an $r$-multidrop line for $\pi_w(X)$.
\end{proposition}
\begin{proof}
First observe that $\pi_w(q_0),\pi_w(q_1),\pi_w(z)$ are collinear, so that $\pi_w(q_1) \in \bbP L$ as well. Notice that $\pi_w(\sigma_r(X)) \subseteq \sigma_r(\pi_w(X))$ because for every $x_1 \vvirg x_r \in X$, we have $w \notin \langle x_1 \vvirg x_r\rangle$, as $\rmR_X(w) > r$; in particular $\pi_w \langle x_1 \vvirg x_r \rangle = \langle \pi_w(x_1) \vvirg \pi_w(x_r) \rangle$. This shows that $\pi_w(q_0),\pi_w(q_1) \in \sigma_r(\pi_w(X))$.

By assumption $\bbP L \not\subseteq \sigma_r(\pi_w(X))$, hence $\bbP L$ is a $r$-multidrop line for $\sigma_r(\pi_w(X))$.
\end{proof}

Proposition \ref{prop: projection generates multidrop} generalizes the example of the rational quartic in $\bbP^3$ described in Section \ref{subsec: rational quartic}. Indeed, let $X$ be a smooth rational quartic in $\bbP^3$, then $X$ is realized as the projection of a rational normal quartic $Y \subseteq \bbP^4$ from a point $w \in \bbP^4$. Let $\pi_w : Y \to X$ be this projection.

If $\uR_Y(w) \leq 2$ then it is easy to verify that $\pi_w(X)$ is not smooth. Therefore, $\rmR_Y(w) = \uR_Y(w) = 3$. Then, there are infinitely many planes $\langle y_1,y_2,y_3\rangle \subseteq \bbP^4$ with $y_1,y_2,y_3 \in Y$ such that $w \in \langle y_1,y_2,y_3\rangle$. The images of $\pi_w(y_1),\pi_w(y_2),\pi_w(y_3)$ are collinear and the line $\bbP L = \langle \pi_w(y_1),\pi_w(y_2),\pi(y_3) \rangle$ is a trisecant line to $X = \pi_w(Y)$.

\section{Submultiplicativity for homogeneous polynomials} \label{section: homogeneous poly}
This section studies submultiplicativity properties of rank of homogeneous polynomials with respect to Veronese varieties. Let $V$ be a vector space of dimension $n+1$ and let $\Sym^dV$ be the space of degree $d$ homogeneous polynomials in $n+1$ variables. The \textit{Veronese variety} $X_d$ is defined as the image of the Veronese embedding 
\[
\nu_d : \bbP V \longrightarrow \bbP \Sym^d V, \qquad [\ell] \mapsto [\ell^d].
\]
If $V$ is $2$-dimensional, then $X_d$ is the rational normal curve $\calC_d \subset \bbP^d$.

For homogeneous polynomials $f \in \Sym^d V$, the $X_d$-rank of the point $[f] \in \bbP \Sym^dV$ is classically referred to as \emph{Waring rank}, and it coincides with the minimum number of $d$-th powers of linear forms needed to express $f$ as their linear combination. In the following, write $f$ both for the element in $\Sym^d V$ and for its projective class $[f]$. 

The first example of rank strict submultiplicativity given in \cite{ChrJenZui:NonMultTensorRank} was the monomial $xy^2$, for which $\rank_{\calC_3} (xy^2) = 3$, but $\rank_{\calC_3^{\times 2}}(xy^2 \otimes xy^2) \leq 8 < 9$. In fact, $\rank_{\calC_3^{\times 2}}(xy^2 \otimes xy^2) = 8$, see \cite{ChFri:TwoThreeQbitRank8}.  In \cite{BalBerChrGes:PartiallySymRkW}, strict submultiplicativity was observed for all monomials of the form $xy^d$, for any $d \geq 3$. In this case, one has $\rmR_{\calC_{d+1}}(xy^d) = d+1$ (see, e.g., \cite{CCG}), but $\rank_{\calC_{d+1}^{\times 2}}(xy^d \otimes xy^d) \leq 4d +1 < (d+1)^2$, see \cite[Proposition 3.5]{BalBerChrGes:PartiallySymRkW}. Note that the form $xy^d$ lies on a tangent line to $\calC_{d+1}$, so $\uR_{\calC_d}(xy^{d}) = 2$ and, therefore, it has border rank smaller than the rank. In particular, these are instances for which Conjecture \ref{conj} holds.

In this section, we study the cases of binary forms and ternary cubics, proving that Conjecture \ref{conj} holds in these cases. In fact, in these cases, having border rank strictly smaller than the rank is equivalent to have strict submultiplicativity for the rank of the second tensor power. However, we believe this to be a low dimensional phenomenon, due to the fact that in these cases the border rank lower bounds are attained by flattening methods, as explained below. We expect that as soon as there are cases of strict submultiplicativity of border rank, then strict submultiplicativity of rank will occur as well, even in examples where rank and border rank coincide. This is what happens in the tensor setting in the construction of \cite{ChrGesJen:BorderRankNonMult} via the multidrop lines described in Section \ref{sec:geometric_cond}.

We introduce briefly some basics about flattening methods. A \textit{flattening} of $\Sym^d V$ is a linear map $\Flat : \Sym^d V \to \Hom(E,F)$ where $E,F$ are vector spaces. One can exploit semicontinuity of matrix rank to obtain lower bounds on $\uR(f)$ from the rank of $\Flat(f)$ as a map from $E$ to $F$. More precisely, from \cite[Proposition 4.1.1]{LanOtt:EqnsSecantVarsVeroneseandOthers}, we have
\begin{equation}\label{eqn: flattening lower bound}
 \uR(f) \geq \frac{\rk (\Flat(f))}{\max\{ \rk( \Flat(\ell^d)) : \ell \in V\}}.
\end{equation}
Since Waring rank and Waring border rank are invariant under the action of the general linear group $GL(V)$ acting on the variables, usually one chooses $\Flat$ to be a $GL(V)$-equivariant map. In \cite{ChrJenZui:NonMultTensorRank}, it was shown that flattening lower bounds are \textit{multiplicative} with respect to tensor products. This directly follows from: (a) the fact that one can combine two flattening maps by taking the Kronecker product of the images and (b) the rank of a linear map is multiplicative under Kronecker product. In \cite[Section 7.2]{BalBerChrGes:PartiallySymRkW}, generalizations to other varieties are given.

\subsection{Binary forms} We recall basic facts of apolarity theory. We refer to \cite{IaKa:book, Ger:notes} for~details.

Let $V$ be a vector space of dimension $2$ and let $V^*$ be its dual. Write $\calC_d = \nu_d(\bbP V)$ for the rational normal curve in $\bbP \Sym^d V$. Let $\{x,y\}$ be a basis for $V$. Then, we identify the symmetric algebra $\Sym^\bullet V$ with the polynomial algebra $\bbC[x,y]$ and $\Sym^\bullet V^*$ with the algebra of differential operators with constant coefficients $\bbC[\partial_x,\partial_y]$, where $\partial_x = \frac{\partial}{\partial x}$ and $\partial_y = \frac{\partial}{\partial y}$. Then, $\Sym^\bullet V^*$ acts on $\Sym^\bullet V$ by contraction with the differential operators:
\[
	\circ : \Sym^\bullet V^* \times \Sym^\bullet V \longrightarrow \Sym^\bullet V, \qquad (\varphi,f) \mapsto \varphi \circ f := \varphi(f).
\]
The \textit{apolar ideal} of a homogeneous polynomial $f \in \Sym^d V$, denoted $f^\perp \subset \Sym^\bullet V^*$, is the ideal of differential operators which annihilate $f$, i.e., $f^\perp = \{\varphi \in \Sym^\bullet V^* ~|~ \varphi \circ f = 0\}$. It is a classical result that if $f$ is a binary form, then 
\[
f^\perp = (\varphi_1, \varphi_2), \qquad \text{ where } \deg(\varphi_1)+\deg(\varphi_2) = \deg(f) + 2. 
\]
If $\deg(\varphi_1) \leq \deg(\varphi_2)$ then $\uR_{\calC_d}(f) = \deg(\phi_1)$. Moreover, if $\varphi_1$ is square-free (as a homogeneous polynomial in $\partial_x,\partial_y$), then $\rank_{\calC_d}(f) = \deg(\varphi_1)$, otherwise, $\rank_{\calC_d}(f) = \deg(\varphi_2)$. As a consequence, the rank of a generic binary form of degree $d$ is $\left\lceil \frac{d+1}{2} \right\rceil$. These results essentially date back to Sylvester \cite{Syl}; for more recent references, see \cite{IaKa:book, CoSe:Binary, BGI}. 

The following result proves a stronger version of Conjecture \ref{conj} in the case of binary forms.

  \begin{theorem}\label{thm:submult_binary}
Let $f \in \Sym^d V$, $\dim V = 2$. Then, $$\brank_{\calC_d}(f) < \rank_{\calC_d}(f) \qquad \text{ if and only if } \qquad \rank_{\calC_d^{\times 2}}(f^{\otimes 2})<\rank_{\calC_d}(f)^2.$$
  \end{theorem}
  \begin{proof}
First, we show that if $\brank_{\calC_d}(f) = \rank_{\calC_d}(f)$ then multiplicativity holds. This is a consequence of multiplicativity of flattening lower bounds. Indeed, for every $e \leq d$, define the following flattening map, called \textit{$e$-th catalecticant}: 
$$\begin{array}{c c c c}
	\cat_e : & \Sym^d V & \longrightarrow & \Hom( \Sym^e V^* , \Sym^{d-e}V) \\	& f & \mapsto & \cat_e(f) : \phi \mapsto \phi \circ f.
\end{array}
$$
Equivariantly, $\cat_e$ is an embedding of $\Sym^d V$ into $\Sym^e V \otimes \Sym^{d-e} V$. 

It is easy to see that if $f$ is a binary form with $\uR(f) = r$, then $\rk(\cat_{e}(f)) \leq r$ with equality for $r-1 \leq e \leq d-(r-1)$ (this is a classical result that goes back to Sylvester \cite{Syl}). In particular, the denominator in \eqref{eqn: flattening lower bound} is one and the border rank of binary forms matches a flattening lower bound.

By multiplicativity of flattening lower bounds, we deduce $\rmR_{\calC_d^{\times 2}}(f \otimes f) = \uR_{\calC_d^{\times 2}}(f \otimes f) = r^2$.

Now, suppose that $\uR_{\calC_d}(f) < \rmR_{\calC_d}(f)$. From the discussion above, $f^\perp = (\phi_1,\phi_2)$ with $\uR_{\calC_d}(f) = \deg(\phi_1) < \deg(\phi_2) = \rmR_{\calC_d}(f)$ and $\rmR_{\calC_d}(f) + \uR_{\calC_d}(f) = d+2$. In particular, if $r_g = \left\lceil \frac{d+1}{2} \right\rceil$ is the generic rank in $\bbP \Sym^d V$, we have $\rmR_{\calC_d}(f) \geq r_g+1$. Write $\rmR_{\calC_d}(f) = (r_g + 1) + \rho$ for some $\rho \geq 0$.

We use induction on $\rho$. 
 
{\it Case $\rho = 0$, i.e., $\rmR_{\calC_d}(f) = r = r_g+1$.} For any $s$, let $Y_s$ be the locus of forms having rank exactly~$s$. 

\begin{quote}
{\bf Claim.} {\it $Y_{r_g}$ is Zariski-open.} 
\begin{proof}[Proof of Claim]
We use \cite[Theorem 11]{CoSe:Binary}. Write $d = 2\delta$ or $d = 2\delta+1$ depending on the parity; then $r_g = \delta+1$. If $d = 2\delta$ is even, $Y_{r_g}$ is the complement of $\sigma_{\delta} (\calC_d)$, hence Zariski-open. If $d = 2\delta +1$ is odd, then $Y_{r_g}$ is the complement of the union $\sigma_{\delta}(\calC_d) \cup Y_{\delta+2}$. Write $Y_{r_g} = \bbP \Sym^d V \setminus (\sigma_{\delta}(\calC_d) \cup Y_{\delta+2}) = (\bbP \Sym^d V \setminus \sigma_{\delta} (\calC_d) ) \setminus Y_{\delta+2}$. Since $\sigma_{\delta}(\calC_d)$ is closed, $\bbP \Sym^d V \setminus \sigma_{\delta}(\calC_d)$ is open and it suffices to show that $Y_{\delta+2}$ is a closed subset of $\bbP \Sym^d V \setminus \sigma_{\delta}(\calC_d)$. This can be shown as follows. Let $h \in Y_{\delta+2}$, namely $\rmR_{\calC_d}(h) = \delta+2$ so that $\uR_{\calC_d}(h) = \delta+1$; the apolar ideal of $h$ is $h^\perp = (\psi_1,\psi_2)$ with $\deg(\psi_1) = \delta+1$ and $\psi_1$ not square-free. Therefore, $Y_{\delta+2}$ is determined by the vanishing of the discriminant of $\psi_1$; indeed, the generator of lowest degree in the apolar ideal of an element of $Y_{\delta+1}$ has degree $\delta+1$ and it is square-free; since the vanishing of the discriminant is a closed condition, we conclude. 
\end{proof}
\end{quote}
  
  Now, let $\ell$ be a linear form such that there exists a scalar $a$ with $\rmR_{\calC_d}(f-a \ell^d) = \rmR_{\calC_d} (f) -1$, namely $f-a \ell^d \in Y_{r_g}$. Since $Y_{r_g}$ is Zariski-open by the Claim, we deduce that $f- \eps \ell^d \in Y_{r_g}$ for a generic choice of $ \eps $. In particular, let $\eps$ be such that $q_0 = f- 2 \eps \ell^d$ and $q_1 = f- \eps \ell^d$ are elements of $Y_{r_g}$. We conclude using the same argument as in \cite{ChrJenZui:NonMultTensorRank}: we write
\begin{equation}\label{eq:CJZdrop}
   f \otimes f = (q_0 + 2 \eps \ell^d) \otimes (q_0 + 2 \eps \ell^d) = q_0 \otimes q_0 + q_1 \otimes \eps\ell^d + \eps\ell^d \otimes q_1
 \end{equation}
and we deduce $\rmR_{\calC_d^{\times 2}}(f \otimes f) \leq (r-1)^2 + 2 (r-1) = r^2 - 1 < \rmR_{\calC_d}(f)^2$.

{\it Case $\rho \geq 1$.} Let $\ell$ be a linear form such that $\rmR_{\calC_d}( f - \ell^d ) = \rmR( f ) -1$. Let $g = f - \ell^d$ and observe that $\rmR_{\calC_d}(g) = r -1 = (r_g + 1) + \rho -1 \geq r_g + 1$. In particular $\rmR_{\calC_d^{\times 2}}(g \otimes g ) \leq  \rmR_{\calC_d}(g)^2 -1 = (r-1)^2 -1 $, by the induction hypothesis. Hence, we write
\[
 f \otimes f = (g - \ell^d) \otimes (g - \ell^d) = g \otimes g - g \otimes \ell^d - \ell^d \otimes g + \ell^d \otimes \ell^d,
\]
and, passing to the rank, we obtain
\[
 \rmR_{\calC_d^{\times 2}} ( f \otimes f) \leq [ (r-1)^2 -1] + (r-1) + (r-1) + 1 = r^2 -1.\vspace{-0,6cm}
\]
\end{proof}

  \subsection{Plane cubics} 
 Also in the case of ternary cubics we prove a stronger version of Conjecture~\ref{conj}. Let $X_3$ be the Veronese surface given by the Veronese embedding of $\bbP V$ in $\bbP \Sym^3 V$, with $\dim V = 3$.
 
There is a complete classification of plane cubics, up to change of coordinates: in other words, the $GL(V)$-orbits in $\Sym^3 V$ are entirely classified, together with ranks and border ranks of their elements, see e.g. \cite{KoganMaza:ComputationCanonicalFormsTernaryCubics,LaTe10:RanksBorderRanks}. We record them as in \cite{CarCatOne:WaringLoci} for convenience:
\begin{enumerate}[(i)]
 \item\label{(0)class} $f = x^3 + y^3 + z^3 + a \cdot xyz$ with $a^3 \neq -3^3,0,6^3$: the generic cubic, with $\uR(f) = \rmR(f) = 4$;
 \item\label{(i)class} $f = x^3 + y^3 + z^3$: the Fermat cubic, with $\uR(f) = \rmR(f) = 3$;
 \item\label{(ii)class} $f = x^2(x-z) + y^2z$: the nodal cubic with $\uR(f) = \rmR(f) = 4$;
 \item\label{(iii)class} $f = x^3+ y^2z$: the cuspidal cubic with $\uR(f) = 3$, $\rmR(f) = 4$;
\item\label{(iv)class} $f = z(x^2 + y^2 + z^2)$: a conic and a secant line, with $\uR(f) = \rmR(f) = 4$;
 \item\label{(v)class} $f = z(x^2 + yz)$: a conic and a tangent line, with $\uR(f) = 3$, $\rmR(f) = 5$;
\item\label{(vi)class} $f = xyz$: three lines intersecting generically, with $\uR(f) = \rmR(f) = 4$;
\item\label{(vii)class} $f = xy(x+y)$: three lines intersecting in a single point, with $\uR(f) = \rmR(f) = 2$;
\item\label{(viii)class} $f = x^2y$: a double line and a transverse line, with $\uR(f) = 2$, $\rmR(f) = 3$;
\item\label{(ix)class} $f = x^3$: a triple line, with $\uR(f) = \rmR(f) = 1$.
 \end{enumerate}
\begin{proposition}\label{prop:submult_planecubics}
 Let $f \in \Sym^3V$, $\dim V = 3$. Then, 
 \[
\rmR_{X_3^{\times 2}}(f\otimes f) < \rank_{X_3}(f)^2 \qquad \text{ if and only if } \qquad \rank_{X_3}(f) < \brank_{X_3}(f).
 \]
  \end{proposition}
  \begin{proof}
  The cases \eqref{(vii)class}, \eqref{(viii)class} and \eqref{(ix)class} are already included in Theorem \ref{thm:submult_binary}. Hence, it suffices to show that $\rmR_{X_3^{\times 2}}(f\otimes f)  = \rank_{X_3}(f)^2$ in cases \eqref{(0)class}, \eqref{(i)class}, \eqref{(ii)class}, \eqref{(iv)class}, \eqref{(vi)class} and $\rmR_{X_3^{\times 2}}(f\otimes f) < \rank_{X_3}(f)^2$ in cases \eqref{(iii)class}, \eqref{(v)class}.

  The equality in cases  \eqref{(0)class}, \eqref{(i)class}, \eqref{(ii)class}, \eqref{(iv)class}, \eqref{(vi)class} follows by multiplicativity of flattening lower bounds, which implies multiplicativity of border rank. For $f \in \Sym^3V$, define the linear map  $f^{\wedge 1}$ given by the following composition
  \[
 f^{\wedge 1} :  V \otimes V^* \xto{\id_V \boxtimes \cat_1(f)} V \otimes \Sym^2 V \xto{\delta}  \Lambda^2 V \otimes V,
  \]
  where the first map is the catalecticant map \emph{augmented} via the identity on a factor $V$ and the second map $\delta :V \otimes \Sym^2 V \to \Lambda^2 V \otimes V$ is the classical Koszul differential \cite[Chapter 17]{Eis:CommutativeAlgebra}. It is a classical fact that $\uR_{X_3}(f) = 4$ if and only if $\rk(f^{\wedge 1}) = 8$; see \cite[Section 3.10]{Lan12:Book}. In particular, the flattening map $\Sym^3 V \to \Hom( V \otimes V^* , \Lambda^2 V \otimes V)$ provides a flattening lower bound in the cases  \eqref{(0)class}, \eqref{(i)class}, \eqref{(ii)class}, \eqref{(iv)class}, \eqref{(vi)class}, which is multiplicative.
  
  Strict submultiplicativity in the case \eqref{(iii)class} follows directly from the submultiplicativity for the monomial $y^2z$. Indeed, we have
  \begin{align*}
\rmR_{X_3^{\times 2}}((x^3 + y^2z) \otimes &(x^3 + y^2z))  \\ &\leq \rmR_{X_3^{\times 2}}(x^3 \otimes x^3) + \rmR_{X_3^{\times 2}}(y^2z \otimes x^3) + \rmR_{X_3^{\times 2}}(x^3 \otimes y^2z) + \rmR_{X_3^{\times 2}}(y^2z \otimes y^2z) \\ & = 1 + 3 + 3 + 8 = 15 < 16 = (\rmR_{X_3}(x^3+y^2z))^2. 
  \end{align*}

Consider the case \eqref{(v)class}, i.e., $f = z(x^2+yz)$. Let, $f_\eps = f - \eps y^3$. Then, $\rk (f_1^{\wedge 1}) = 8$, showing $\uR_{X_3}(f_\eps) = 4$, for a generic choice of $\eps$. From the classification, observe that $\rmR_{X_3}(g_\eps) = 4$ as well. In particular, there exists a coefficient $\eps$ such that both $f_\eps$ and $f_{2\eps}$ have rank $4$. Proceeding as in~\eqref{eq:CJZdrop}, we conclude $\rank_{X_3}(f^{\otimes 2})\leq 24$.
\end{proof}

\subsection{A lower bound for rank of tensor product of bivariate monomials} 
We conclude this section providing a lower bound on the rank of the tensor product of two monomials in two variables. We use a method introduced in \cite{CCG}. 
	
Let $m = x^ay^b$, with $a \leq b$. It is known that $\rank_{\calC_{a+b}}(m) = b+1$ (see e.g.  \cite[Remark 24]{BGI}, \cite[Proposition~3.1]{CCG}  and \cite[Corollary 2]{RS11:RankSymmetricForm}) and that $\underline{\rank}_{\calC_{a+b}}(m) = a+1$ (see e.g. \cite[Theorem 1.11]{LaTe10:RanksBorderRanks} and \cite[Corollary 2]{RS11:RankSymmetricForm}). By Theorem \ref{thm:submult_binary}, we have that $\rank_{\calC_{a+b}^{\times 2}}(m^{\otimes 2}) = \rank_{\calC_{a+b}}(m)^2$ if and only if $a = b$. From \cite[Proposition 4.3]{BalBerChrGes:PartiallySymRkW}, we have the lower bound $\rank_{\calC_d^{\times 2}}(m \otimes m) \geq 2d - 1$ when $m = xy^{d-1}$. We provide a lower bound for any $x^ay^b$ with $a < b$. 

The strategy of the proof allows us to formulate the result for the tensor product of two binary monomials, not necessarily equal.

\begin{proposition}\label{prop: bound monomials}
	Let $m_1 = x_1^{a_1}y_1^{b_1}$ and $m_2 = x_2^{a_2}y_2^{b_2}$ with $a_1 \leq b_1$ and $a_2\leq b_2$. Then, 
	\[
	\rank_{\calC_{a_1+b_1} \times \calC_{a_2 + b_2}}(m_1 \otimes m_2) \geq \max \{ (a_1+1)(b_2+1), (a_2+1)(b_1+1)\}.  
\]
\end{proposition}
\begin{proof}
We show $\rmR_{\calC_{a_1+b_1} \times \calC_{a_2 + b_2}}(m_1 \otimes m_2) \geq (a_1+1)(b_2+1)$.

Regard $m_1 \otimes m_2$ in the bi-graded polynomial ring $S = \Bbbk[x_1,y_1]\otimes \Bbbk[x_2,y_2] = \Bbbk[x_1,y_1;x_2,y_2]$, where $\deg(x_1) = \deg(y_1) = (1,0)$ and $\deg(x_2) = \deg(y_2) = (0,1)$. Let $S_{(i,j)}$ be the $\Bbbk$-vector space of homogeneous polynomials of bi-degree $(i,j)$. The apolar ideal of $m_1 \otimes m_2$ is
\[
(m_1 \otimes m_2)^\perp = (x_1^{a_1+1},y_1^{b_1+1},x_2^{a_2+1},y_2^{b_2+1}). 
\]
From the multigraded version of the Apolarity Lemma (see e.g. \cite{GalRanVil:VarsApolarSubschToricSurfaces, Galazka:MultigradedApolarity, Bernardi_2011, BALLICO20196, Bernardi_2008}), the rank of $m_1 \otimes m_2 $ coincides with the minimal cardinality of a reduced set of points $\bbX \subset \bbP^1 \times \bbP^1$ whose bi-graded defining ideal $I_\bbX$ is contained in $(m_1 \otimes m_2)^\perp$. For any bi-graded ideal $I \subset S$, we denote $I_{(i,j)} = I \cap S_{(i,j)}$. The Hilbert function of $S/I$ in degree $(i,j)$ is $\HF_{S/I}(i,j) := \dim_\Bbbk S_{(i,j)}/I_{(i,j)}$. We refer to \cite{SV06} for basic properties of the bi-graded Hilbert function of ideals of points in multi-projective space. For convenience, we just recall two facts about the Hilbert function of ideals of points that we will use in the rest of the proof:
\begin{enumerate}
	\item the Hilbert function of $S/I_\bbX$ is strictly increasing until it gets constant along the $i$-th row $\left(\HF_{S/I_\bbX}(i,j)\right)_{j \geq 0}$ and along the $j$-th column $\left(\HF_{S/I_\bbX}(i,j)\right)_{i \geq 0}$;
	\item if $i,j \gg 0$, then $\HF_{S/I_\bbX}(i,j) = |\bbX|$.
\end{enumerate}

Let $\bbX \subseteq \bbP^1 \times \bbP^1$ be a set of points with $I_\bbX \subseteq (m_1 \otimes m_2)^\perp$. Let $\bbX' = \bbX \setminus \{ x_1 = 0\}$: we have $I_{\bbX'} = I_\bbX : (x_1)$ and $x_1$ is not a zero-divisor in $S/ I_{\bbX'}$. This implies that, for $i \gg 1$ and for $j \geq 1$:
\[
|\bbX| \geq |\bbX'| \geq  {\rmHF}_{S/I_\bbX'}(i,j) = \sum_{k=1}^i {\rm HF}_{S/I_{\bbX'} + (x_1)}(k,j). 
\]
Now, since $I_\bbX \subset (m_1 \otimes m_2)^\perp$, we have  $I_{\bbX'} \subset (m_1 \otimes m_2)^\perp : (x_1)$, so that 
\[
\sum_{k \geq 0} {\rmHF}_{S/I_{\bbX'} + (x_1)}(k,j) \geq \sum_{k \geq 0}{\rmHF}_{S/\left((m_1 \otimes m_2)^\perp:(x_{1}) + (x_{1})\right)}(k,j).
\]
Let $J = (m_1 \otimes m_2)^\perp : (x_1) + (x_1) = (x_1,y_1^{b+1},x_2^{a+1},y_2^{b+1})$. We have 
\[
S/J \simeq \Bbbk[x_1,y_1] / (x_1,y_1^{b_1+1}) \otimes \Bbbk[x_2,y_2] / (x_2^{a_2+1},y_2^{b_2+1}).
\]
Therefore,
\[
\sum_{k \geq 0} {\rmHF}_{S/J}(k,j) = \sum_{k \geq 0}\dim\left(\Bbbk[x_1,y_1] / (x_1,y_1^{b_1+1})\right)_{(k)} \cdot \dim\left(\Bbbk[x_2,y_2] / (x_2^{a_2+1},y_2^{b_2+1})\right)_{(j)}. 
\]

Let $j = a_2$, so that the right hand side is~$(b_1+1)(a_2+1)$. We conclude that
\[
\rank_{\calC_{a_1+b_1} \times \calC_{a_2 + b_2}}(m_1 \otimes m_2) \geq |\bbX| \geq |\bbX'| \geq (b_1+1)(a_2+1). 
\]
Exchanging the roles of $m_1$ and $m_2$, we get $\rmR_{\calC_{a_1+b_1} \times \calC_{a_2 + b_2}} ( m_1 \otimes m_2) \geq (b_2+1) (a_1+1)$.
\end{proof}
The bound of Proposition \ref{prop: bound monomials} is far from being sharp: recall that $\rmR_{\bbC_3 \times \bbC_3}(x^2 y \otimes x^2 y) = 8$ \cite{ChFri:TwoThreeQbitRank8,BalBerChrGes:PartiallySymRkW}, while the lower bound from Proposition \ref{prop: bound monomials} is just $6$.

\section{Minimal decompositions of products and products of minimal decompositions}\label{sec:product_decomp}

In this section, we focus on cases in which multiplicativity of rank holds. In particular, we ask whether minimal decompositions of a tensor product always arise as tensor products of minimal decompositions of the factors. 

We consider varieties $X_1,\ldots,X_k$, with $X_i \subset \bbP V_i$, for all $i = 1,\ldots,k$. Given sets of points $S_1,\ldots,S_k$ with $S_i \subset \bbP V_i$, denote by $S_1\times\ldots\times S_k$ both the cartesian product in $\bbP V_1 \times \ldots \bbP V_k$ and its image with respect to the Segre embedding $\bbP V_1 \times \ldots \times \bbP V_k \rightarrow \bbP (V_1\ootimes V_k)$, i.e., the set $\{z_1\otimes\ldots\otimes z_k ~|~ z_i \in S_i\}$.

First, we provide an immediate result on minimality of product decompositions.
\begin{lemma}\label{lemma: product is minimal wrt inclusion}
 For $i = 1,\ldots,k$, let $X_i \subset \bbP V_i$ be a variety, let $p_i \in X_i$ and let $S_i \subseteq X_i$ be a non redundant set of points spanning $p_i$, namely $p_i \in \langle S_i \rangle$ and no proper subset of $S_i$ spans $p_i$. Let $S = S_1 \times \cdots \times S_k$. Then, there is no proper subset $T \subsetneq S\subset \bbP (V_1\ootimes V_k)$ such that $p_1 \otimes \cdots \otimes p_k \in \langle T \rangle$.
\end{lemma}
\begin{proof}
 We proceed by induction on $k$. Let $k = 2$. Write $S_1 = \{ a_1 \vvirg a_{s_1}\}$ and $S_2 = \{ b_1 \vvirg b_{s_2}\}$. Without loss of generality, we may assume $p_1 = \sum_i a_i$ and $p_2 = \sum_i b_i$. So $p_1 \otimes p_2$ can be regarded as an element of the space of matrices $\langle S_1 \rangle \otimes \langle S_2 \rangle \subseteq V_1 \otimes V_2$. In particular, the set $S_1 \times S_2 = \{ a_i \otimes b_j: i= 1 \vvirg s_1, j = 1 \vvirg s_2\}$ gives a basis for $\langle S_1 \rangle \otimes \langle S_2 \rangle $ and we can choose coordinates so that $a_i \otimes b_j$ is represented by the matrix having the $(i,j)$-th entry equal to $1$ and zero elsewhere, and $p_1 \otimes p_2$ is represented by the matrix having $1$ in every entry. We conclude that no proper $T \subsetneq S_1 \times S_2$ can span $p_1 \otimes p_2$ because every element of $\langle T \rangle$ has a zero entry.
 
 If $k \geq 2$, the statement follow by induction regarding $p_1\ootimes p_k$ as $p_1 \otimes (p_2\ootimes p_k)$.
\end{proof}

\begin{notation}
	Given a Cartier divisor $D$ and a set of points $A$, we denote by $\Res_D(A)$ the \textit{residual set of points with respect to $A$}, namely the set $A \smallsetminus (A \cap D)$. In particular, if $\calI_A$ and $\calI_D$ are the ideal sheaves defining $A$ and $D$, respectively, then $\calI_{\Res_D(A)} = \calI_A : \calI_D$.
\end{notation}

 In Theorem \ref{thm: identifiability product via embeddings}, we will use a slight variant of \cite[Lemma 2.5]{BalBerChrGes:PartiallySymRkW}. Given a variety $X$, a very ample line bundle $\calL$ on $X$, and a set of points $S \subseteq X$, we say that $S$ \emph{imposes independent conditions on the sections of $\calL$} if $h^1(\calI_S \otimes \calL) = h^1(\calL)$ or equivalently if the restriction map $H^0(\calL) \to H^0(\calL|_S)$ is surjective.
\begin{lemma}\label{lemma: bbcg revisited} 
  Let $X$ be a variety and let $\calL$ be a very ample line bundle on $X$. Let $V = H^0(\calL)^*$ and identify $X$ with its embedding in $\bbP V$. Let $p \in \bbP V$ and let $A,B \subset X$ be non redundant sets of points spanning $p$ in $\bbP V$. Assume $h^1(\mathcal I_{B}\otimes \calL) = h^1(\calL)$. Let $D$ be an effective Cartier divisor on $X$ such that $\Res_D(A)\cap\Res_D(B) = \emptyset$. If $h^1 (\calI_{ \Res_{D} (A \cup B)} \otimes \calL(-D)) = 0$ then $A \cup B \subseteq D$.
\end{lemma}
\begin{proof}
 The proof is essentially the same as \cite[Lemma 2.5]{BalBerChrGes:PartiallySymRkW} with the bundle $\calL$ replacing the bundle $\calO_{\bbP ^N}(1)$. The condition $h^1(\calI_B(1)) = 0$ in \cite[Lemma 2.5]{BalBerChrGes:PartiallySymRkW} is replaced by $h^1(\calI_B \otimes \calL) = h^1(\calL)$ and the same argument provides the proof.
\end{proof}
Recall the definition of \textit{identifiability}: given $X \subseteq \bbP ^N$ and $p \in \bbP^N$ with $\rmR_{X}(p)=r$, one says that $p$ is \emph{identifiable} if there is a unique set of $r$ points of $X$ whose span contains $p$. We refer to \cite{Chiantini_quantum} for a basic introduction to identifiability problems.
\begin{theorem}\label{thm: identifiability product via embeddings}
Let $X_1 \vvirg X_k$ be irreducible projective varieties. For every $i = 1 \vvirg k$, let $\calL_i$ be a very ample line bundle on $X_i$ and identify $X_i$ with the embedded subvariety in $V_i = H^0(\calL_i)^*$ defined by the sections of $\calL_i$. Let $\calM_i$ and $\calN_i$ be line bundles on $X_i$ with an isomorphism $\calM_i \otimes \calN_i \simeq \calL_i$. Let $p_i \in \bbP V_i$ and $S_i \subseteq X_i$ be a set of points evincing $\rmR_{X_i}(p_i)$, for every $i = 1 \vvirg k$. Then:
 \begin{enumerate}[(i)]
\item\label{(i)statement} If for every $i = 1 \vvirg k$, $S_i$ imposes independent conditions on the sections of $\calN_i,\calM_i$ and $\calL_i$ and in addition $h^1(M_i) = 0$, then 
  \[
  \rmR_{X_1 \ttimes X_k}(p_1 \ootimes p_k) = \textprod_{i = 1}^k \rmR_{X_i}(p_i).
  \]

\item\label{(ii)statement} If for every $i =1 \vvirg k$, $\calI_{S_i} \otimes \calN_i$ has no base points outside $S_i$, then $S_1 \ttimes S_k$ is the unique set of points evincing $\rmR_{X_1 \ttimes X_k}(p_1 \otimes \cdots \otimes p_k)$. In particular, the $p_i$'s are identifiable as well as $p_1\ootimes p_k$.
 \end{enumerate}
\end{theorem}
\begin{proof}
Set $r_i = \rmR_{X_i}(p_i)$. Let $A \subseteq X_1 \ttimes X_k$ be a set of points enhancing the rank of $p_1 \ootimes p_k$. We show that $\deg(A) \geq \prod r_i$. Write $\calL = \calL_1 \ootimes \calL_k$, regarded as a line bundle on $X_1 \ttimes X_k$ and similarly for $\calM$ and $\calN$. 

Let $B = S_1 \ttimes S_k$. By Lemma \ref{lemma: product is minimal wrt inclusion}, there is no proper subset of $B$ spanning $p_1 \ootimes p_k$. We will show $H^0(\calI_A \otimes \calN) \subseteq H^0(\calI_B \otimes \calN)$. If $H^0(\calI_A \otimes \calN) = 0$, this statement is true. Otherwise, let $D \in H^0(\calI_A \otimes \calN)$ and identify $D$ with the divisor in $X_1 \ttimes X_k$ that it defines.

We verify the hypotheses of Lemma \ref{lemma: bbcg revisited} to show $A \cup B \subseteq D$:
\begin{itemize}
\dotitem $B$ imposes independent conditions on $\calL$: this is straightforward from K\"unneth's formula as $H^0(\calL_1 \ootimes \calL_k) = \bigotimes H^0 ( \calL_i)$ and similarly for the restrictions to $S$;
\dotitem $\Res_D(A) \cap \Res_D(B) = \emptyset$, because $\Res_D(A) = \emptyset$ since $A \subseteq D$;
\dotitem $h^1(\calI_{\Res_D(A \cup B)} \otimes \calL(-D)) = 0$: notice $\calL(-D) = \calM$ because $\calL = \calM \otimes \calN$ \cite[Prop. II.6.13]{h}; moreover $\Res_D(A \cup B) = \Res_D(B) \subseteq B$ and $B$ imposes independent conditions on $\calM$, which guarantees $h^1(\calI_{\Res_D(A\cup B)} \otimes \calL(-D)) \leq  h^1(\calI_{B} \otimes \calM) = h^1(\calM)$, where the right-hand-side is equal to $0$, again by K\"unneth's formula.
\end{itemize}

Applying Lemma \ref{lemma: bbcg revisited}, deduce $A \cup B \subseteq D$ and so $B \subseteq D$, and we deduce $H^0(\calI_A\otimes\calN) \subseteq H^0(\calI_B\otimes\calN)$.

Again, $B$ imposes independent conditions on $\calN$, namely the restriction map $H^0(\calN) \to H^0( \calN|_B)$ is surjective and one has the exact sequence
\[
 0 \to H^0( \calI_B \otimes \calN) \to H^0(\calN) \to H^0( \calN|_B) \to 0.
\]
This provides $h^0(\calI_B \otimes \calN)  = h^0(\calN) - \deg(B)  = h^0(\calN) - \textprod_i r_i $. On the other hand $ h^0(\calI_B \otimes \calN) \geq h^0(\calI_A \otimes \calN) \geq h^0(\calN) - \deg(A)$. Hence, $\deg(A) \geq \prod_i r_i$ and part \eqref{(i)statement} of the statement holds.

It remains to show that if the hypothesis of \eqref{(ii)statement} is satisfied, then $A = B$. From the first part of the proof we have $\deg(A) = \deg(B)$. Then, $H^0(\calI_A \otimes \calN) = H^0(\calI_B \otimes \calN)$. Indeed: $h^0(\calI_B~\otimes~\calN) = h^0(\calN) - \deg(B)$, because $B$ imposes independent conditions on $\calN$ and, on the other hand, $h^0 ( \calI_A \otimes \calN) = h^0(\calN) - \deg(A) + h^1(\calI_A \otimes \calN) - h^1(\calN)$; notice $h^1(\calI_A \otimes \calN) - h^1(\calN) \geq 0$ because $H^1 ( \calI_A \otimes \calN) \to H^1(\calN)$ is surjective. This shows $h^0(\calI_B \otimes \calN) = h^0(\calI_A \otimes \calN)$ and, since from the first part of the proof $H^0(\calI_A\otimes \calN) \subseteq H^0(\calI_B \otimes \calN)$, the equality follows.

Now, by assumption $\calI_{S_i} \otimes \calN_i$ has no base points out of $S_i$, and therefore $\calI_B \otimes \calN$ has no base points outside of $B$: this shows $A \subseteq B$ and since they have the same degree equality holds. 
\end{proof} 
 
Theorem \ref{thm: identifiability product via embeddings} can be applied to Veronese varieties as follows.
 
\begin{corollary}\label{prop: product identifiability rat nor curves}
Let $i = 1 \vvirg k$. For every $i$, let $d_i \geq 1$, $n_i \geq 1$ and let $f_i \in \bbP \Sym^{d_i} \bbC^{n_i+1}$ be elements with $\rmR_{\nu_{d_i}(\bbP^{n_i})}(f_i) = r_i \leq \lceil d_i/2 \rceil$. Let $S_i \subseteq \bbP^{n_i}$ be a set of $r_i$ points such that $p_i \in \langle \nu_{d_i}(S_i) \rangle$. Then $\rmR_{\nu_{d_1 \vvirg d_k}(\bbP^{n_1} \ttimes \bbP^{n_k} )} (f_1\ootimes f_k) = r_1 \cdots r_k$ and $S_1 \times \cdots \times S_k$ is the unique set of points in ${\bbP^{n_1}} \ttimes \bbP^{n_k}$ such that $p_1 \otimes \cdots \otimes p_k \in \langle \nu_{d_1 \vvirg d_k} (S_1 \times \cdots \times S_k) \rangle$.
\end{corollary}

We point out that in general, already for points of rank two, there are decompositions of the product not arising from decompositions of the single factors. 
\begin{example}\label{ex:NonProdMinimalDecomp}
 Let $X_1 \subseteq \bbP^{N_1}$ be a variety; suppose there is a $2$-dimensional linear space $\bbP A$ such that $\bbP A \cap X_1$ contains at least four points of intersection $a_1,a_2,a_3,a_4 \in X_1$. Let $p_1 \in \langle a_1, a_2\rangle$ be a point such that $\uR_{X_1}(p_1) = 2$. 
 
 Let $X_2 \subseteq \bbP ^{N_2}$ be a variety, let $b_1,b_2 \in X_2$ be two points and let $p_2 \in \langle b_1,b_2 \rangle$ be a point such that $\rmR_{X_2}(p_2) = 2$. Write $p_1 = a_1+a_2$ and $p_2 = b_1 + b_2$.

Suppose that the condition described in Theorem \ref{thm: rank 2 x rank 2 classification} does not hold, so that $\rmR_{X_1 \times X_2}(p_1 \otimes p_2)= 4$. Then $\{ a_1 , a_2\} \times \{ b_1,b_2\} = \{a_1\otimes b_1,a_2 \otimes b_1, a_1 \otimes b_2, a_2 \otimes b_2\}$ is a minimal set of four points $X_1 \times X_2$ spanning $p_1 \otimes p_2$.
 
 We determine a second set of four points as follows. After a suitable choice of coordinates, write $a_4 = a_1+a_2-a_3$. Define 
 \[
\begin{array}{llll}
a'_1 = a_1, & a'_2 = a_2, & a'_3 = a_3, & a'_4 = a_4,\\
b'_1 = b_1, & b'_2 = b_1, & b'_3 = b_2, & b'_4 = b_2,\\ 
\end{array}
\]
Then 
\begin{align*}
 \textsum_i a'_i \otimes b'_i & = a_1 \otimes b_1 + a_2 \otimes b_1 + a_3 \otimes b_2 + (a_1+a_2 - a_3) \otimes b_2 \\
 &= (a_1 + a_2)\otimes (b_1 + b_2) = p_1 \otimes p_2.
\end{align*}
This shows that $S = \{ a'_1 \otimes b'_1 \vvirg a'_4 \otimes b'_4\}$ is a set of four points of $X_1 \times X_2$ spanning $p_1 \otimes p_2$. The set $S$ is not of the form $S_1 \times S_2$ for some $S_1 \subseteq X_1$, $S_2 \subseteq X_2$.
\end{example}

 \bibliographystyle{amsalpha}
\begin{small}
\bibliography{bibMultisec}
\end{small}
 
 \end{document}

%% file: mypreamble.tex
\usepackage[english]{babel}
\usepackage{amsmath}
\usepackage{amssymb}
\usepackage{amsthm}
\usepackage{latexsym}
\usepackage{mathtools}
\usepackage{thmtools}
\usepackage{amsfonts}
\usepackage{mathrsfs}
\usepackage{textcomp}
\usepackage[T1]{fontenc}
\usepackage{graphicx}
\usepackage{setspace}
\usepackage{nicefrac}
\usepackage{indentfirst}
\usepackage{enumerate}
\usepackage{wasysym}
\usepackage{upgreek}
\usepackage[pdfpagelabels,hyperindex=false]{hyperref}
\usepackage{paralist}
\usepackage{xcolor}
\hypersetup{
    colorlinks,
    linkcolor={red!50!black},
    citecolor={green!50!black},
    urlcolor={blue!80!black}
}
\usepackage{etoolbox}
\usepackage{mathdots}
\usepackage{wrapfig}
\usepackage{floatflt}
\usepackage{tensor} 
\usepackage{parskip}
\usepackage{lscape}
\usepackage{stmaryrd}
\usepackage[enableskew]{youngtab}
\usepackage{ytableau}
\usepackage{pifont}

\usepackage{tikz}
\usetikzlibrary{arrows,matrix}
\usetikzlibrary{positioning}
\usetikzlibrary{decorations}

\usepackage[all,cmtip]{xy}
\usepackage[labelfont=bf, font={small}]{caption}[2005/07/16]

\usepackage{xcolor}
\definecolor{red}{rgb}{1,0,0}

\newcommand{\xto}[1]{\xrightarrow{\phantom{a}{#1}{\phantom{a}}}}

\newcommand{\vvirg}{ , \dots , }
\newcommand{\ootimes}{ \otimes \cdots \otimes }

\newcommand{\ttimes}{ \times \cdots \times }

\newcommand{\dotitem}{ \item[$\cdot$]}

\newcommand{\textsum}{{\textstyle \sum}}
\newcommand{\textprod}{{\textstyle \prod}}

\newcommand{\bfz}{\mathbf{z}}

\newcommand{\calC}{\mathcal{C}}

\newcommand{\calI}{\mathcal{I}}

\newcommand{\calL}{\mathcal{L}}
\newcommand{\calM}{\mathcal{M}}
\newcommand{\calN}{\mathcal{N}}
\newcommand{\calO}{\mathcal{O}}

\newcommand{\calR}{\mathcal{R}}
\newcommand{\calS}{\mathcal{S}}

\newcommand{\bbC}{\mathbb{C}}

\newcommand{\bbP}{\mathbb{P}}

\newcommand{\bbX}{\mathbb{X}}

\newcommand{\frakS}{\mathfrak{S}}

\newcommand{\rmR}{\mathrm{R}}

\renewcommand{\phi}{\varphi}
\newcommand{\eps}{\varepsilon}
\renewcommand{\theta}{\vartheta}

\newcommand{\dashto}{\dashrightarrow}

\renewcommand{\tilde}[1]{\widetilde{#1}}

\renewcommand{\bar}[1]{\overline{#1}}

\newcommand{\id}{\mathrm{id}}

\DeclareMathOperator{\Hom}{Hom}

\DeclareMathOperator{\Sym}{Sym}

\DeclareMathOperator{\mult}{mult}

\newcommand{\rmHF}{\mathrm{HF}}

\DeclareMathOperator{\Res}{Res}

\newcommand{\fillwidthof}[3][c]
	{%
		\parbox
		{%
			\widthof{#2}%
		}%
		{%
			\ifx#1c%
				\centering#3%
			\else\ifx#1l%
				#3\hfill%
			\else\ifx#1r%
				\hfill#3%
			\fi\fi\fi%
		}%
	}%

\def\mylettrine#1#2 {\lettrine{#1}{#2}\space}

\newcommand{\partinto}[1][]{\smash{\mathord{\mathchoice{%
  \xymatrix@=0.4em@1{%
  \ar@{|-}[rr]_-*--{\scriptstyle #1}
  &*{\phantom{\scriptstyle{#1}}}&}
}{
  \xymatrix@=0.25em@1{%
  \ar@{|-}[rr]_-*--{\scriptstyle #1}
  &*{\phantom{\scriptstyle{#1}}}&}
}{
  \xymatrix@=0.2em@1{%
  \ar@{|-}[rr]_-*--{\scriptscriptstyle #1}
  &*{\phantom{\scriptscriptstyle{#1}}}&}
}{}}}}
\newcommand{\partintonosmash}[1][]{\mathord{\mathchoice{%
  \xymatrix@=0.4em@1{%
  \ar@{|-}[rr]_-*--{\scriptstyle #1}
  &*{\phantom{\scriptstyle{#1}}}&}
}{
  \xymatrix@=0.25em@1{%
  \ar@{|-}[rr]_-*--{\scriptstyle #1}
  &*{\phantom{\scriptstyle{#1}}}&}
}{
  \xymatrix@=0.2em@1{%
  \ar@{|-}[rr]_-*--{\scriptscriptstyle #1}
  &*{\phantom{\scriptscriptstyle{#1}}}&}
}{}}}
\newcommand{\partintostar}[1][]{\smash{\mathord{\mathchoice{%
  \xymatrix@=0.4em@1{%
  \ar@{|-}[rr]_-*--{\scriptstyle #1}^-*--{\scriptstyle \ast}
  &*{\phantom{\scriptstyle{#1}}}&}
}{
  \xymatrix@=0.25em@1{%
  \ar@{|-}[rr]_-*--{\scriptstyle #1}^-*--{\scriptstyle \ast}
  &*{\phantom{\scriptstyle{#1}}}&}
}{
  \xymatrix@=0.2em@1{%
  \ar@{|-}[rr]_-*--{\scriptscriptstyle #1}^-*--{\scriptstyle \ast}
  &*{\phantom{\scriptscriptstyle{#1}}}&}
}{}}}}
\newcommand{\partintostarnosmash}[1][]{\mathord{\mathchoice{%
  \xymatrix@=0.4em@1{%
  \ar@{|-}[rr]_-*--{\scriptstyle #1}^-*--{\scriptstyle \ast}
  &*{\phantom{\scriptstyle{#1}}}&}
}{
  \xymatrix@=0.25em@1{%
  \ar@{|-}[rr]_-*--{\scriptstyle #1}^-*--{\scriptstyle \ast}
  &*{\phantom{\scriptstyle{#1}}}&}
}{
  \xymatrix@=0.2em@1{%
  \ar@{|-}[rr]_-*--{\scriptscriptstyle #1}^-*--{\scriptstyle \ast}
  &*{\phantom{\scriptscriptstyle{#1}}}&}
}{}}}

